\documentclass[a4paper,10pt]{amsart}
\usepackage[foot]{amsaddr}
\makeatletter
\renewcommand{\email}[2][]{%
  \ifx\emails\@empty\relax\else{\g@addto@macro\emails{,\space}}\fi%
  \@ifnotempty{#1}{\g@addto@macro\emails{\textrm{(#1)}\space}}%
  \g@addto@macro\emails{#2}%
}
\makeatother
\usepackage{a4wide}
\usepackage{stmaryrd}
\usepackage{nicefrac}
\usepackage{paralist}
\usepackage{hyperref}
\usepackage{amssymb}
\usepackage[normalem]{ulem}
\hypersetup{pdfstartview=XYZ}
\hypersetup{
  unicode = true,                           
  pdftoolbar = true,                        
  pdfmenubar = true,                        
  pdffitwindow = true,                      
  pdfauthor = {Chaumont-Frelet, Ern, Lemaire, Valentin},                                                  
  pdftitle = {Bridging the Multiscale Hybrid-Mixed and Multiscale Hybrid High-Order methods},             
  pdfnewwindow = true,                      
  colorlinks = true,                        
  linkcolor = blue,                         
  citecolor = magenta,                      
  filecolor = violet,                       
  urlcolor = violet                         
}

\newcommand{\defi}{\mathrel{\mathop:}=}
\newcommand{\ifed}{=\mathrel{\mathop:}}

\newcommand{\VV}{\mathcal V}

\newcommand{\TT}{\mathcal T}
\newcommand{\dT}{\partial \mathcal  T}
\newcommand{\FF}{\mathcal F}
\newcommand{\FFi}{\mathcal F^{\rm int}}
\newcommand{\FFe}{\mathcal F^{\rm bnd}}

\newcommand{\nn}{\boldsymbol n}

\newcommand{\HH}{\boldsymbol H}

\newcommand{\sig}{\boldsymbol \sigma}

\newcommand{\PP}{\mathbb P}

\newcommand{\grad}{\boldsymbol \nabla}
\renewcommand{\div}{\grad \cdot}

\newcommand{\tH}{\widetilde H}

\newcommand{\jump}[1]{\llbracket #1 \rrbracket}

\newcommand{\ddiv}{\operatorname{div}}

\newcommand{\hu}{\widehat u}
\newcommand{\hv}{\widehat v}
\newcommand{\hU}{\widehat U}
\newcommand{\hSigma}{\widehat \Sigma}

\newtheorem{theorem}{Theorem}
\newtheorem{lemma}[theorem]{Lemma}

\newtheorem{remark}[theorem]{Remark}

\numberwithin{equation}{section}
\numberwithin{theorem}{section}

\newcommand{\pb}{{\rm p}}
\newcommand{\db}{{\rm d}}
\newcommand{\uN}{^{\textsc{n}}}
\newcommand{\us}{^{\textsc{s}}}
\newcommand{\dK}{\partial K}
\newcommand{\calU}{\mathcal{U}}

\newcommand{\calV}{\mathcal{V}}
\newcommand{\uMHM}{^{\textsc{mhm}}}
\newcommand{\uHHO}{^{\textsc{hho}}}

\title[Bridging the MHM and MsHHO methods]{Bridging the Multiscale Hybrid-Mixed and Multiscale Hybrid High-Order methods}

\author{Th\'eophile Chaumont-Frelet$^\dag$}
\address[$\dag$]{Inria, Univ.~C\^ote d'Azur, CNRS, UMR 7351 - Laboratoire J.~A. Dieudonn\'e, F-06000 Nice, France}
\email[$\dag$]{theophile.chaumont@inria.fr}
\author{Alexandre Ern$^\sharp$}
\address[$\sharp$]{CERMICS, \'Ecole des Ponts, F-77455 Marne-la-Vall\'ee Cedex 2 \& Inria Paris, F-75589 Paris, France}
\email[$\sharp$]{alexandre.ern@enpc.fr}
\author{Simon Lemaire$^\flat$}
\address[$\flat$]{Inria, Univ.~Lille, CNRS, UMR 8524 - Laboratoire Paul Painlev\'e, F-59000 Lille, France}
\email[$\flat$]{simon.lemaire@inria.fr}
\author{Fr\'ed\'eric Valentin$^\ddag$}
\address[$\ddag$]{LNCC, Petr\'opolis - RJ, Brazil \& Inria, F-06000 Nice, France}
\email[$\ddag$]{valentin@lncc.br}

\begin{document}

\begin{abstract}
We establish the equivalence between the Multiscale Hybrid-Mixed (MHM)
and the Multiscale Hybrid High-Order (MsHHO) methods for a variable diffusion problem with
piecewise polynomial source term. Under the idealized assumption
that the local problems defining the multiscale basis functions
are exactly solved, we prove that the equivalence holds for general polytopal
(coarse) meshes and arbitrary approximation orders. We also leverage the interchange of properties to perform a unified convergence analysis, as well as to improve on both methods. 
\end{abstract}

\maketitle

\section{Introduction}
\label{intro}

The tremendous development of massively parallel architectures in the last decade has
led to a revision of what is expected from computational simulators, which must embed
asynchronous and communication-avoiding algorithms.
In such a scenario where precision and robustness remain fundamental
properties, but algorithms must take full advantage of the new architectures, numerical methods
built upon the ``divide-and-conquer" philosophy fulfill these requirements better than standard
methods operating in a monolithic fashion on the different scales of the problem at hand. 
Among the vast literature on the subject, driven by domain decomposition methodologies
(see, e.g., \cite{TosWid05} for a survey), multiscale numerical methods emerge as an attractive
option to efficiently handle problems with highly heterogeneous coefficients, as well as multi-query scenarios in which the problem solution must be computed for a large number of source terms. These scenarios may arise when considering highly oscillatory, nonlinear, time-dependent models, or within optimization algorithms when solving problems featuring PDE-based constraints, or in models including stochastic processes, to cite a few.

The development of multiscale methods started with the seminal work \cite{BabOsb83}. Important advances were then provided in~\cite{Hug95,HuFeM:98} (cf.~also~\cite{BreBriFraMalRog92,BreRus94}, and the unifying viewpoint of~\cite{BreFraHugRus97b}) and in \cite{HouWu:97,HouWuCai99}, laying the ground, respectively, for the Variational Multiscale method, and for the Multiscale Finite Element (MsFE) method.
Overall, the common idea behind these multiscale methods is to consider basis functions especially designed so as to
upscale to an overlying coarse mesh the sub-mesh variations of the model. 
Particularly appealing is the fact that the multiscale basis functions are defined by entirely
independent problems. From this viewpoint, multiscale
numerical methods may also be seen as a (non-iterative) domain decomposition technique
\cite{GloWhe88}.
Since the pioneering works on multiscale methods, a large number of improvements and new approaches have been proposed.
In the MsFE context (see~\cite{EfHou:09} for a survey), one can cite the oversampling technique of~\cite{EfHoW:00}, as well as the Petrov--Galerkin variant of~\cite{HouWZ:04} (see also~\cite{AraBarFraVal09}), or the high-order method of~\cite{AlBri:05} (see also~\cite{HesZZ:14}).
  More recent research directions focus on reducing and possibly eliminating the cell resonance error. In this vein, one can cite the Generalized MsFE method~\cite{EfGaH:13}, or the Local Orthogonal Decomposition approach~\cite{HenPe:13,MalPet14}. 
  Hybridization has also been investigated in the pioneering work~\cite{ArPWY:07} on multiscale mortar mixed finite element methods (see also the multiscale mortar multipoint flux mixed finite element method of~\cite{WheXueYot12}). These ideas have been adapted later on in the context of (multiscale) Discontinuous Galerkin methods, leading to the Multiscale Hybridizable Discontinuous Galerkin (MsHDG) method of~\cite{EfeLazShi15} (cf.~also the multiscale Weak Galerkin method of~\cite{MuWaY:16}, devised along the same principles in the spirit of the Generalized MsFE method). Interestingly, this latter approach enables to relax the constraints between the mortar space and the polynomial spaces used in the mesh cells.

Recently, two families of hybrid multiscale numerical methods that are applicable on general meshes have been proposed, namely the Multiscale Hybrid-Mixed (MHM) and the Multiscale Hybrid High-Order (MsHHO)
methods. The MHM method has been first introduced in \cite{HarParVal13}, and further analyzed in
\cite{ArHPV:13,PaVaV:17,BarJaiParVal20} (see also \cite{HarVal16} for an abstract setting), whereas the MsHHO method has
been proposed in \cite{CiErL:19a,CiErL:19b}, as an extension of the HHO method first introduced in \cite{DPELe:14,DPErn:15} (cf.~also \cite{DPELe:16}).
The MHM method relates to the mixed multiscale finite element method proposed in \cite{CheHou02}, as well as to the subgrid upscaling method of \cite{ArbBoy06} (see \cite[Sec.~5.1.2]{HarVal16} for further details). The MsHHO method generalizes to arbitrary polynomial orders the low-order nonconforming multiscale methods of~\cite{LBLLo:13,LBLLo:14}. The polynomial unknowns attached to the mesh interfaces in the MsHHO method play a different role with respect to the (coarse) interface unknowns of the MsHDG method of~\cite{EfeLazShi15}. The fundamental difference between these two approaches is that the MsHDG method is based on local Dirichlet problems (the interface unknowns are then the traces of the solution), whereas the MsHHO method is based on local Neumann problems (the interface unknowns are then the coarse moments of the traces of the solution). Notice that the MHM method is also based on local Neumann problems.
Note that similar ideas have been developed in the conforming framework in the context of BEM-based FEM~\cite{CoLaP:09,Weiss:19}.

The MHM and MsHHO methods substantially differ in their construction. Picking the Poisson equation as an example, the MHM method hinges on
the primal hybrid formulation analyzed in \cite{RavTho77b}. As a consequence, while the local problems are defined as coercive Neumann problems, the global upscaled linear system is of saddle-point type, involving face unknowns that are the normal fluxes through the mesh faces (also the Neumann data for the local problems, up to the sign), plus one  degree of freedom per mesh cell that enforces a local balance between the normal fluxes and the source term. Notice that the (global) saddle-point structure of the MHM method can be  equivalently replaced  by a sequence of positive-definite linear systems as shown recently  in \cite{MadSar21}. On the other hand, the MsHHO method is directly built upon the primal formulation of the problem. As a consequence, the local (Neumann) problems are defined as constrained minimization problems, and as such exhibit a saddle-point structure. On the contrary, the global upscaled linear system is coercive, and only involves face unknowns that are the coarse moments of the traces of the solution at interfaces. Notice that, as opposed to the MHM method, the MsHHO method also uses cell unknowns (that are locally eliminable from the global upscaled linear system), which are associated with basis functions solving local problems with nonzero source terms. As such, the MsHHO method is naturally suited to deal with multi-query scenarios.
  

In this work, we revisit the MHM and MsHHO methods and we prove an equivalence result
between their solutions. Notice that such a relationship is not straightforward since, at first glance,
the two methods exhibit structures that are genuinely different.
Nonetheless, we demonstrate that such an equivalence holds under the assumption
that the source term of the continuous problem is piecewise polynomial (cf.~Theorem~\ref{th:equiv}).
For this equivalence to hold, we make the idealized assumption that the local problems defining the multiscale basis functions are exactly solved. The corresponding methods are then referred to as {\em one-level} (cf.~Remark~\ref{rem:sec_lev} for some insight on the equivalence between two-level methods). Leveraging this equivalence result, the present work also contributes to derive, in a unified fashion, an energy-norm error estimate that is valid for both methods (cf.~Theorem~\ref{th:err.est}). More specifically, 
\begin{itemize}
\item  in the MHM framework, this result is a refined version (especially in the tracking of the dependency with respect to the diffusion coefficient) of the results in~\cite{ArHPV:13};
\item in the MsHHO framework, this result is new and is complementary to the homogenization-based error estimate derived in~\cite{CiErL:19a}.
\end{itemize}
We also explore these stimulating results to transfer properties
proved for one method to the other, and to reveal how the interplay between the methods can
drive advances for both. Notably, we show that
\begin{itemize}
\item the MHM method can be adapted to deal with multi-query scenarios (cf.~Section~\ref{ssse:MHM});
\item  the MsHHO method can be recast as a purely face-based method, in the sense that it can be alternatively defined without using cell unknowns (cf.~Section~\ref{sec:face-based}).
\end{itemize}

The outline of the article is as follows. Section \ref{model} introduces the model problem,
the partition, the notation and a number of useful tools. We present the MHM method in Section
\ref{sec:mhm}, and the MsHHO method in Section \ref{MsHHO-method}. The equivalence
result is stated in Section \ref{equivalence}, along with some further properties and remarks. The energy-norm error estimate is proved in Section~\ref{se:conv}. The solution strategies for both methods are discussed in Section \ref{basis-ddm}, leveraging the equivalence result at hand to propose enhancements for both methods.
Finally, some conclusions are drawn in Section \ref{concl}.

\section{Setting}
\label{model}

In this section, we present the setting, introduce the partition, and define useful broken
spaces on this partition.

\subsection{Model problem}
We consider an open polytopal domain $\Omega\subset\mathbb R^d$, $d=2$ or $3$, with boundary $\partial\Omega$. Given $f:\Omega\to\mathbb{R}$, we seek a function $u:\Omega\to\mathbb{R}$ such that
\begin{equation}
\label{eq_laplace_strong}
\left \{
\begin{alignedat}{2}
-\div(\mathbb{A}\grad u) &= f &\quad&\text{ in $\Omega$}\,,
\\
u &= 0 &\quad&\text{ on $\partial \Omega$}\,.
\end{alignedat}
\right .
\end{equation}
We assume that the diffusion coefficient $\mathbb{A}\in L^\infty(\Omega;\mathbb{R}^{d\times d})$ is symmetric and uniformly elliptic, and that the source term $f$ is in $L^2(\Omega)$.
Problem \eqref{eq_laplace_strong} admits the following weak form:
find $u \in H^1_0(\Omega)$ such that
\begin{equation}
\label{eq_laplace_weak}
(\mathbb{A}\grad u,\grad v)_{\Omega} = (f,v)_{\Omega}\qquad\text{for all $v \in H^1_0(\Omega)$}\,,
\end{equation}
where $(\cdot,\cdot)_{D}$ denotes the $L^2(D;\mathbb{R}^\ell)$, $\ell\in\{1,d\}$, inner product for any measurable set $D\subset\overline{\Omega}$. It is well-established that Problem~\eqref{eq_laplace_weak} admits a unique solution.

\subsection{Partition} \label{sse:part}

The domain $\Omega$ is partitioned into a (coarse) mesh $\TT_H$, that consists of
polytopal (open) cells $K$ with diameter $H_K$, and we set
$H \defi \max_{K \in \TT_H} H_K$. In practice, both the MHM and MsHHO methods consider a fine submesh (characterized by a mesh-size $h\ll H$) to compute the local basis functions, but this finer mesh is not needed in the present discussion since we will assume that the local problems defining the basis functions are exactly solved.
The mesh faces $F$ of $\TT_H$ are collected in the set $\FF_H$, and this set is partitioned into the subset of internal faces (or interfaces) $\FFi_H$ and the subset of boundary faces $\FFe_H$.
The mesh faces are defined to be planar, i.e., every mesh face $F\in\FF_H$ is supported by an affine hyperplane $\mathcal{H}_F$ (recall that the mesh cells have planar faces since they are polytopes).
For an interface $F \in \FFi_H$, we have
\begin{equation}
\label{eq_face_int}
F = \partial K_+ \cap \partial K_-\cap \mathcal{H}_F\,,
\end{equation}
for two cells $K_\pm \in \TT_H$; for a boundary face $F \in \FFe_H$, we have
\begin{equation}
\label{eq_face_ext}
F = \partial K \cap \partial \Omega\cap \mathcal{H}_F\,,
\end{equation}
for one cell $K\in\TT_H$.
We denote by $\partial\TT_H$ the skeleton of the mesh $\TT_H$, defined by $\partial\TT_H\defi\bigcup_{K\in\TT_H}\{\partial K\}$.
Given $K\in\TT_H$, we denote by $\FF_K$ the set of its faces, and by $\nn_K$ the unit outward-pointing vector normal to its boundary (whose restriction to the face $F\in\FF_K$ is the constant vector denoted by $\nn_{K,F}$). We associate with each face $F \in \FF_H$ a unit normal vector $\nn_F$ whose orientation is fixed, with the convention that $\nn_F \defi \nn_{\Omega\mid F}$ if $F \in \FFe_H$, where $\nn_{\Omega}$ is the unit outward-pointing vector normal to $\partial \Omega$.
\begin{remark}[On the notion of face]
  Some minor variations are encountered in the literature regarding the notion
of face in a polytopal mesh, depending on whether the faces are required or not to be
planar, and whether they are genuinely or only loosely defined. In the (polytopal) Discontinuous Galerkin literature~\cite{DiPEr:12,CaDGH:17}, faces are (genuinely) defined by $F=\partial K_+\cap\partial K_-$ (or $F=\partial K\cap\partial\Omega$), thus allowing for nonplanarity. In the HHO literature, faces are always required to be planar, so that one can define a constant normal vector $\nn_F$ to every face $F\in\FF_H$. Variations however exist on how to define them. In the original work~\cite{DPELe:14} on HHO methods, faces are defined loosely by $F\subseteq\partial K_+\cap\partial K_-\cap\mathcal{H}_F$ (or $F\subseteq\partial K\cap\partial\Omega\cap\mathcal{H}_F$); on the contrary, in~\cite[Sec.~1.2.1]{CiErP:21} and in the present work, faces are genuinely defined by $F=\partial K_+\cap\partial K_-\cap\mathcal{H}_F$ (or $F=\partial K\cap\partial\Omega\cap\mathcal{H}_F$). Notice that the latter (genuine) definition, as opposed to the loose one, does not allow for the case of several coplanar faces that would be shared by two cells (or a cell and the boundary). It is however more precise, which is the reason why we have chosen to adopt it in this work. Remark also that, as opposed to the one in~\cite{DPELe:14} (or in \cite[Def.~1.4]{DPDro:20}), the present definition does not require explicitly that faces are connected sets. Of course, the methods we study here are also applicable under the setting of~\cite{DPELe:14}.
\end{remark}

\subsection{Infinite-dimensional broken spaces} \label{sec:infinite_broken}
 
We first define the broken space of piecewise smooth functions on $\TT_H$:
\begin{equation}
H^1(\TT_H) \defi \left \{
v \in L^2(\Omega) \, : \, v_K \in H^1(K) \quad \forall K \in \TT_H
\right \}\,,
\end{equation}
where we let $v_D\defi v_{\mid D}$. For any $v \in H^1(\TT_H)$, we define the jump $\jump{v}_F$ of $v$ across $F \in \FF_H$ by
\begin{equation}
\jump{v}_F \defi 
v_{K_+\mid F}\,(\nn_{K_+,F}\cdot\nn_F) + v_{K_-\mid F}\,(\nn_{K_-,F}\cdot\nn_F)
\end{equation}
if $F \subseteq \partial K_+ \cap \partial K_-$ is an interface, and simply by
\begin{equation}
  \jump{v}_F \defi v_{K\mid F}
\end{equation}
if $F \subseteq \partial K \cap \partial \Omega$ is a boundary face. We also define the broken gradient operator $\grad_H: H^1(\TT_H) \to L^2(\Omega;\mathbb{R}^d)$ such that, for any $v\in H^1(\TT_H)$,
\begin{equation} \label{eq:def_grad_H}
\left (\grad_H v\right )_{\mid K} \defi \grad v_K \quad \text{for all } K \in \TT_H\,.
\end{equation}
We next introduce the space of piecewise smooth functions on $\TT_H$ whose broken (weighted) flux belongs to $\HH(\ddiv,\Omega)$:
\begin{equation}
\VV(\TT_H;\ddiv,\Omega) \defi \left \{
v \in H^1(\TT_H) \, : \, \mathbb{A}\grad_H v \in \HH(\ddiv,\Omega)
\right \}.
\end{equation}
We will see below that the MHM and MsHHO methods produce a discrete
solution that sits in the space $\VV(\TT_H;\ddiv,\Omega)$;
notice that $\VV(\TT_H;\ddiv,\Omega)\subset H^1(\TT_H)\not\subset H^1(\Omega)$.
We now define the two ``skeletal'' spaces
%
\begin{equation}
\Sigma_{0}(\dT_H) \defi \left \{
z\defi (z_{\partial K} )_{K\in\TT_H} \in \prod_{K \in \TT_H}  H^{1/2}(\partial K)
\, \left | \,
\begin{array}{l}
\exists\, w(z) \in H^1_0(\Omega)\text{ s.t.}
\\
z_{\partial K} =w_K(z)_{\mid\partial K}  \; \forall K \in \TT_H
\end{array}
\right .
\right \}\,,
\end{equation}
and
%
\begin{equation}
\Lambda(\dT_H) \defi \left \{
\mu\defi(\mu_{\partial K})_{K\in\TT_H} \in \prod_{K \in \TT_H} H^{-1/2}(\partial K)
\, \left | \,
\begin{array}{l}
\exists \, \sig(\mu) \in \HH(\ddiv,\Omega)\text{ s.t.}
\\
\mu_{\partial K}=\sig_K(\mu)_{\mid\partial K} \cdot \nn_K \; \forall K \in \TT_H
\end{array}
\right .
\right \}\,.
\end{equation}
(Recall that the subscript $K$ refers to the restriction to $K$.)
Letting $\langle \cdot,\cdot \rangle_{\partial K}$ stand for the duality pairing between $H^{-1/2}(\partial K)$ and $H^{1/2}(\partial K)$, we define the following pairing, for all $\mu \in\prod_{K \in \TT_H} H^{-1/2}(\partial K)$ and all $z\in\prod_{K \in \TT_H} H^{1/2}(\partial K)$,
\begin{equation}
\langle \mu,z\rangle_{\dT_H} \defi \sum_{K \in \TT_H} \langle \mu_{\partial K},z_{\partial K} \rangle_{\partial K}\,,
\end{equation}
so that for all $\mu \in \Lambda(\dT_H)$ and all $z\in \Sigma_{0}(\dT_H)$, recalling that $\boldsymbol{\sigma}(\mu)\in\boldsymbol{H}({\rm div},\Omega)$ and $w(z)\in H^1_0(\Omega)$, we have
\begin{equation}\label{zeroflux}
\langle \mu,z\rangle_{\dT_H} = \sum_{K \in \TT_H} \Big( (\div \sig(\mu),w(z))_K + (\sig(\mu),\grad w(z))_K\Big)=0\,.
\end{equation}

\subsection{Finite-dimensional broken spaces} 

Let $q\in\mathbb N$ denote a given polynomial degree.
The space of piecewise ($d$-variate) polynomial functions on $\TT_H$ of total degree up to $q$ is denoted by 
\begin{equation}
\PP^q(\TT_H) \defi \left \{
v \in L^2(\Omega) \; : \; v_K \in \PP^q(K) \quad \forall K \in \TT_H
\right \}\,,
\end{equation}
whereas the space of piecewise ($(d-1)$-variate) polynomial functions on $\FF_H$ of total degree up to $q$ is denoted by
\begin{equation}
\PP^q(\FF_H) \defi\left \{
v \in L^2\left(\bigcup_{F\in\FF_H}F\right) \; : \; v_F \in \PP^q(F) \quad \forall F \in \FF_H
\right \}\,,
\end{equation}
and its subset incorporating homogeneous boundary conditions by
\begin{gather}
\PP^q_0(\FF_H) \defi\left \{v\in\PP^q(\FF_H)\,:\,v_F = 0 \quad\forall F\in\FFe_H\right\}\,.
\end{gather}
For all $K\in\TT_H$, we also define  the local space of piecewise ($(d-1)$-variate) polynomial functions on $\FF_K$ of total degree up to $q$ as follows:
\begin{gather}
\PP^q(\FF_K) \defi\left \{ v \in L^2(\partial K) \; : \; v_F \in \PP^q(F) \quad\forall F \in \FF_K
\right \}\,.
\end{gather}
We consider the following finite-dimensional proper subspace of $\Lambda(\dT_H)$:
\begin{equation}
\Lambda^q(\dT_H)\defi \{\mu\in \Lambda(\dT_H)\, : \,
\mu_{\dK}\in \mathbb{P}^q(\FF_K) \; \forall K\in\TT_H\}\,. 
\end{equation}
Notice that for every interface $F\in\FFi_H$ with $F \subseteq \partial K_+ \cap \partial K_-$, as a consequence of~\eqref{zeroflux}, we
have $\mu_{\partial K_+\mid F} + \mu_{\partial K_-\mid F}=0$ for all $\mu \in \Lambda^q(\dT_H)$.
We also define, for any integer $m\geq 0$, the spaces
\begin{equation} \label{eq:def_calU} 
\left\{ \begin{aligned}
&\calU^{m,q}(K) \defi \left \{
v \in H^1(K) \, : \, \div(\mathbb{A}\grad v) \in \PP^{m}(K), \quad
\mathbb{A}\grad v_{\mid\partial K} \cdot \nn_{K} \in \PP^q(\FF_K)
\right \}\,,\;\forall K\in\TT_H\,, \\
&\calU^{m,q}(\TT_H) \defi \left \{
v \in H^1(\TT_H) \, : \, v_K \in \calU^{m,q}(K) \quad \forall K \in \TT_H\right \}.
\end{aligned}\right.
\end{equation}
To alleviate the notation, we shall drop the superscript $m$ when considering $m = q-1$ for $q\geq 1$, and write $\calU^q(K)$ and $\calU^q(\TT_H)$ in place of $\calU^{q-1,q}(K)$ and $\calU^{q-1,q}(\TT_H)$, respectively.

We finally introduce the space of ``weakly $H^1_0(\Omega)$" functions on $\TT_H$:
\begin{equation}
\tH^{1,q}_0(\TT_H) \defi \left \{
v \in H^1(\TT_H) \, : \, (\jump{v}_F,p)_F= 0 \quad \forall \,p \in \PP^q(F),\; \forall F\in\FF_H\right \}\,.
\end{equation}
Equivalently, we have
\begin{equation} \label{eq:tH1_Lambda}
\tH^{1,q}_0(\TT_H) = \left \{
v \in H^1(\TT_H) \, : \, \langle \mu, v\rangle_{\dT_H}= 0 \quad \forall \mu \in \Lambda^q(\dT_H)\right \}\,.
\end{equation}

\section{The MHM method}
\label{sec:mhm}

%
Let us first set
\begin{equation}\left\{\begin{aligned}
H^1(K)^\perp &\defi \left \{
v \in H^1(K) \, : \, (v,1)_K = 0\right \}\,, \quad \forall K\in\TT_H\,, \\
H^1(\TT_H)^\perp &\defi \left \{
v \in H^1(\TT_H) \, : \,
(v_K,1)_K = 0 \; \forall K \in \TT_H\right \}\,. 
\end{aligned}\right.
\end{equation}
For integers $m,q\in\mathbb{N}$, we also 
define the subspaces $\calU^{m,q}(K)^\perp\defi \{ v\in \calU^{m,q}(K)\, : \, (v,1)_K = 0 \}$ for all $K\in\TT_H$ and $\calU^{m,q}(\TT_H)^\perp\defi 
\{ v\in \calU^{m,q}(\TT_H)\, : \, (v_K,1)_K = 0\; \forall K \in \TT_H \}$.

Let $K\in\TT_H$, and consider the two local operators 
\begin{equation}
T_K\uN: H^{-\frac12}(\partial K) \to H^1(K)^\perp\,, \qquad
T_K\us: L^2(K) \to H^1(K)^\perp\,.
\end{equation} 
For all $\mu_{\dK} \in H^{-\frac12}(\partial K)$
and all $g_K \in L^2(K)$, $T_K\uN(\mu_{\dK})$ and $T_K\us(g_K)$ are the unique elements
in $H^1(K)^\perp$ such that 
\begin{equation}
\label{eq:local_mhm_operators}
\left\{\begin{aligned}
(\mathbb{A}\grad T_K\uN(\mu_{\dK}),\grad v)_{K} &= \langle \mu_{\dK}, v \rangle_{\dK}\,, \\
(\mathbb{A}\grad T_K\us(g_K),\grad v)_{K} &= (g_K,v)_{K} \,,
\end{aligned}\right.
\qquad \forall v \in H^1(K)^\perp\,.
\end{equation}
The superscripts in the operators indicate that $T\uN_K$ lifts a (Neumann) normal flux
and $T\us_K$ lifts a source term. Elementary arguments show that 
\begin{subequations} \label{eq:prop_T}
\begin{alignat}{2}
&-\div (\mathbb{A}\grad T_K\uN (\mu_{\dK}))=-\frac{1}{|K|}\langle\mu_{\dK},1\rangle_{\dK}\;\text{in }K\,,&\quad&\mathbb{A}\grad T_K\uN (\mu_{\dK}) \cdot \nn_K= \mu_{\partial K}\;\text{on }\partial K\,, \label{eq:prop_TuN}\\
&-\div (\mathbb{A}\grad T_K\us (g_K)) = g_K-\frac{1}{|K|}(g_K,1)_K\;\text{in }K\,,&\quad&\mathbb{A}\grad T_K\us (g_K)\cdot\nn_K=0\;\text{on }\partial K\,.\label{eq:prop_Tus}
\end{alignat}
\end{subequations}
It is convenient to define the following global versions of the above lifting operators:
\begin{equation}
T\uN: \Lambda(\dT_H) \to H^1(\TT_H)^\perp\,, \qquad
T\us: L^2(\Omega) \to H^1(\TT_H)^\perp\,.
\end{equation} 
For all $\mu \in \Lambda(\dT_H)$
and all $g \in L^2(\Omega)$, we set 
\begin{equation}
T\uN(\mu)_{\mid K}\defi T\uN_K(\mu_{\dK})\,,\qquad
T\us(g)_{\mid K}\defi T_K\us(g_{K})\,. 
\end{equation}
Equivalently, and recalling the definition~\eqref{eq:def_grad_H} of the 
broken gradient operator, we have
\begin{equation}
\label{eq:global_mhm_operators}
\left\{\begin{aligned}
(\mathbb{A}\grad_H T\uN(\mu),\grad_H v)_{\Omega} &= \langle \mu, v \rangle_{\dT_H}\,, \\
(\mathbb{A}\grad_H T\us(g),\grad_H v)_{\Omega} &= (g,v)_{\Omega} \,,
\end{aligned}\right.
\qquad \forall v \in H^1(\TT_H)^\perp\,,
\end{equation}
which results from summing~\eqref{eq:local_mhm_operators} cell-wise.
We remark that the solution $u\in H^1_0(\Omega)$ to Problem~\eqref{eq_laplace_weak} satisfies
\begin{equation}\label{cont-sol}
  u = u^0 + T\uN(\lambda) + T\us(f) \,,
\end{equation}
where $(u^0,\lambda)\in\mathbb{P}^0(\TT_H)\times\Lambda(\dT_H)$ solve
\begin{subequations}
  \begin{alignat}{2}
    \langle\lambda, v^0\rangle_{\dT_H} &= -(f,v^0)_{\Omega} & \quad &\forall v^0 \in \PP^0(\TT_H)\,,
    \label{cont-a} \\
    \langle\mu, u^0\rangle_{\dT_H} + \langle\mu, T\uN(\lambda)\rangle_{\dT_H}  &= -\langle\mu,T\us(f)\rangle_{\dT_H} & \quad & \forall \mu \in \Lambda(\dT_H)\,.
    \label{cont-b}
  \end{alignat}
\end{subequations}
\noindent
Notice that, owing to~\eqref{eq:global_mhm_operators} and to the fact that $\mathbb{A}$ is symmetric, we have $\langle\mu,T\us(f)\rangle_{\dT_H}=(f,T\uN(\mu))_{\Omega}$.

Let $k\in\mathbb{N}$ be a given polynomial degree.
The  MHM method~\cite{ArHPV:13} reads as follows: Find $(u^0_H,\lambda_H) \in \PP^0(\TT_H) \times \Lambda^k(\dT_H)$ such that
\begin{subequations}\label{mhm}
\begin{alignat}{2}
\langle\lambda_H, v^0_H\rangle_{\dT_H} &= -(f,v^0_H)_{\Omega} & \quad &\forall v^0_H \in \PP^0(\TT_H)\,,
\label{mhm-a} \\
\langle\mu_H, u^0_H\rangle_{\dT_H} + \langle\mu_H, T\uN(\lambda_H)\rangle_{\dT_H}  &= -\langle\mu_H,T\us(f)\rangle_{\dT_H} & \quad & \forall \mu_H \in \Lambda^k(\dT_H)\,,
\label{mhm-b}
\end{alignat}
\end{subequations}
and the  MHM solution is then defined by
\begin{equation}
\label{mhm-sol}
u_H\uMHM \defi u^0_H + T\uN(\lambda_H) + T\us(f) \,.
\end{equation}
The well-posedness of Problem~\eqref{mhm} is established in~\cite[Theorem 3.2]{ArHPV:13}.
Notice that we also have, on the discrete level, $\langle\mu_H,T\us(f)\rangle_{\dT_H}=(f,T\uN(\mu_H))_{\Omega}$.

\begin{lemma}[Characterization of the MHM solution~\eqref{mhm-sol}] \label{lem:semi-mhm}
Let $u_H\uMHM$ be defined by~\eqref{mhm-sol}. Then, \textup{(i)} 
$(\mathbb{A}\grad_H u_H\uMHM\,|_{\dK}) \cdot \nn_K\in\mathbb{P}^k(\FF_K)$ for all $K\in\TT_H$ and $u_H\uMHM \in \tH^{1,k}_0(\TT_H)$; \textup{(ii)} 
$u_H\uMHM \in \calV(\TT_H;\ddiv,\Omega)$ and
$-\div(\mathbb{A}\grad_H u_H\uMHM) = f$ in $\Omega$.
\end{lemma}

\begin{proof}
By~\eqref{mhm-sol} and~\eqref{eq:prop_T}, we infer that for all $K\in\TT_H$,
\begin{equation}\label{eq:fl}
\mathbb{A}\grad_H u_H\uMHM\,|_{\dK} \cdot \nn_K
= \mathbb{A}\grad T_K\uN(\lambda_{H\mid\dK})\cdot \nn_K
+ \mathbb{A}\grad T_K\us(f_{K})\cdot \nn_K = \lambda_{H\mid\dK} \in \mathbb{P}^k(\FF_K)\,.
\end{equation}
That $u_H\uMHM \in \tH^{1,k}_0(\TT_H)$ follows from the characterization~\eqref{eq:tH1_Lambda}
of $\tH^{1,k}_0(\TT_H)$ and \eqref{mhm-b}. Now, to prove that $u_H\uMHM \in \calV(\TT_H;\ddiv,\Omega)$,
we need to show that $\mathbb{A}\grad_H u_H\uMHM \in \HH(\ddiv,\Omega)$.
Owing to~\eqref{eq:prop_T}, we infer that for all $K\in\TT_H$,
\begin{align}
\div(\mathbb{A}\grad_H u_H\uMHM)_{\mid K} &= \div(\mathbb{A}\grad T_K\uN(\lambda_{H\mid\partial K})) +
\div(\mathbb{A}\grad T_K\us(f_{K})) \nonumber \\
&= \frac{1}{|K|}\langle \lambda_{\dK},1\rangle_{\dK} - f_{K} + \frac{1}{|K|}(f_{K},1)_K = -f_{K}\in L^2(K)\,, \label{eq:identity_mhm}
\end{align}
where the last equality follows from~\eqref{mhm-a}. This shows that 
${\mathbb{A}\grad_H u_H\uMHM}\,|_K \in \HH(\ddiv,K)$ for all $K\in\TT_H$. Moreover, \eqref{eq:fl} shows that $\mathbb{A}\grad_H u_H\uMHM\,|_{\dK}\cdot\nn_K$ can be localized to each face of $K$
and, since for every interface $F \subseteq \partial K_+ \cap \partial K_-$, 
$\lambda_{\partial K_+\mid F} + \lambda_{\partial K_-\mid F}=0$, we infer that $\jump{\mathbb{A}\grad_H u_H\uMHM}_{F}\cdot\nn_F=0$ on $F$.
It results that $\mathbb{A}\grad_H u_H\uMHM \in \HH(\ddiv,\Omega)$. Finally,
$-\div(\mathbb{A}\grad_H u_H\uMHM) = f$ in $\Omega$ follows from~\eqref{eq:identity_mhm} since
$K\in\TT_H$ is arbitrary. 
\end{proof}

Let us take a closer look at the MHM method \eqref{mhm}-\eqref{mhm-sol}.
First, we observe that since $T\uN(\lambda_H)\in \calU^{0,k}(\TT_H)^\perp$,
this function is computable from a finite-dimensional calculation.
The same holds for the right-hand side of \eqref{mhm-b} since 
$\langle\mu_H,T\us(f)\rangle_{\dT_H}=(f,T\uN(\mu_H))_{\Omega}$.
However, the situation is different in \eqref{mhm-sol} for $T\us(f)$. 
One needs indeed to define, so as to fully explicit the (one-level) method, an approximation of this function that is also computable from a finite-dimensional calculation.
For this reason, the original MHM method defined by~\eqref{mhm}-\eqref{mhm-sol} can be viewed as {\em semi-explicit},
whereas a {\em fully explicit} version of it is obtained after approximating $T\us(f)$. Among various possibilities (cf.~Remark~\ref{re:rhs} for an example of an alternative definition),
perhaps the simplest one is to choose an integer $m\ge0$, project $f\in L^2(\Omega)$
onto the finite-dimensional subspace $\PP^m(\TT_H)$, and compute $T\us(\Pi^m_H(f))$, where $\Pi^m_H$ is the $L^2$-orthogonal projector onto $\PP^m(\TT_H)$. This leads to the fully explicit MHM solution
\begin{equation}
\label{mhm-sol-alt}
u_H\uMHM \defi u^0_H + T\uN(\lambda_H) + T\us(\Pi^m_H(f)) \,,
\end{equation}
where the pair
$(u^0_H,\lambda_H) \in \PP^0(\TT_H) \times \Lambda^k(\dT_H)$ now solves
\begin{subequations}\label{mhm-alt}
\begin{alignat}{2}
\langle\lambda_H, v^0_H\rangle_{\dT_H} &= -(f,v^0_H)_{\Omega} & \quad &\forall v^0_H \in \PP^0(\TT_H)\,,
\label{mhm-alt-a} \\
\langle\mu_H, u^0_H\rangle_{\dT_H} + \langle\mu_H, T\uN(\lambda_H)\rangle_{\dT_H}  &= -(\Pi^m_H(f),T\uN(\mu_H))_{\Omega} & \quad & \forall \mu_H \in \Lambda^k(\dT_H)\,.
\label{mhm-alt-b}
\end{alignat}
\end{subequations}
We notice in particular that in~\eqref{mhm-sol-alt} we have
$T\uN(\lambda_H) \in \calU^{0,k}(\TT_H)^\perp\subseteq \calU^{m,k}(\TT_H)^\perp$
and $T\us(\Pi^m_H(f))\in \calU^{m,0}(\TT_H)^\perp\subseteq \calU^{m,k}(\TT_H)^\perp$. Thus, all the quantities involved in~\eqref{mhm-sol-alt}-\eqref{mhm-alt} are members of the space $\calU^{m,k}(\TT_H)$.
Adapting the arguments of the proof of Lemma~\ref{lem:semi-mhm} leads
to the following result.

\begin{lemma}[Characterization of the MHM solution~\eqref{mhm-sol-alt}] \label{lem:fully-mhm}
Let $u_H\uMHM$ be defined by~\eqref{mhm-sol-alt}. Then, \textup{(i)} 
$u_H\uMHM \in \calU^{m,k}(\TT_H)\cap\tH^{1,k}_0(\TT_H)$; \textup{(ii)} 
$u_H\uMHM \in \calV(\TT_H;\ddiv,\Omega)$ and
$-\div(\mathbb{A}\grad_H u_H\uMHM) = \Pi^{m}_H(f)$ in $\Omega$.
\end{lemma}

\section{The MsHHO method}
\label{MsHHO-method}

Let again $k\in\mathbb{N}$ be a given polynomial degree, and let $m\geq 0$ be an integer. 
The MsHHO method hinges on the following set of discrete unknowns:
\begin{equation}
\hU_H^{m,k} \defi \PP^{m}(\TT_H) \times \PP^k(\FF_H)\,,
\end{equation}
which is composed of cell and face degrees of freedom (one can also consider the case $m=-1$, so that the method is based on face unknowns only; cf.~Remark~\ref{re:mo}). The standard MsHHO method, referred to as mixed-order MsHHO method in~\cite{CiErL:19a}, corresponds to the case $m=k-1$ for $k\geq 1$. 
For all $K \in \TT_H$, we let
$\hv_K\defi(v_{K},v_{\FF_K}) \in \hU_K^{m,k}\defi \PP^{m}(K) \times \PP^k(\FF_K) $ denote the local counterpart of $\hv_H\defi(v_{\TT_H},v_{\FF_H})\in\hU_H^{m,k}$. For all $F\in\FF_H$, $v_F\in \PP^k(F)$ is defined by $v_F\defi v_{\FF_H\mid F}$. Notice that $v_F=v_{\FF_{K_+}\mid F}=v_{\FF_{K_-}\mid F}$ if $F\subseteq\partial K_+\cap\partial K_-$ is an interface, and $v_F=v_{\FF_K\mid F}$ if $F\subseteq\partial K\cap\partial\Omega$ is a boundary face.

The MsHHO method is based on the following local reconstruction operator:
For all $K \in \TT_H$ and all $\hv_K \in \hU_K^{m,k}$, there exists a unique
function $r_K(\hv_K) \in \calU^{m,k}(K)$ (recall that $\calU^{m,k}(K)$ is defined in~\eqref{eq:def_calU}) such that
\begin{subequations} \label{eq:recons} \begin{align} \label{eq:recons.a}
(\mathbb{A}\grad r_K(\hv_K) , \grad w)_K
&=
-(v_K, \div(\mathbb{A}\grad w))_K
+(v_{\FF_K}, \mathbb{A}\grad w \cdot \nn_K)_{\partial K}
\quad \forall w \in \calU^{m,k}(K)\,, \\
\label{eq:recons.b}
(r_K(\hv_K),1)_{\partial K} &= (v_{\FF_K},1)_{\partial K}\,.
\end{align}\end{subequations}
Notice that the usual choice of closure relation for $r_K(\hv_K)$ is $(r_K(\hv_K),1)_K=(v_K,1)_K$.
The operator $r_K$ is the (local) reconstruction operator associated with the finite element
\begin{gather} \label{eq:triple_MsHHO}
\left(K,\,\calU^{m,k}(K),\,\hSigma_K\right)\,,
\end{gather}
with the set of degrees of freedom $\hSigma_K:\calU^{m,k}(K)\to\hU_K^{m,k}$ such that 
$\hSigma_K(v)\defi\left(\Pi^{m}_K(v),\Pi^k_{\FF_K}(v)\right)$ for all $v\in\calU^{m,k}(K)$, where $\Pi^{m}_K$ and $\Pi^k_{\FF_K}$ are the $L^2$-orthogonal projectors onto, respectively, $\PP^{m}(K)$ and $\PP^k(\FF_K)$. For further use, we also define $\Pi^k_F$ to be the $L^2$-orthogonal projector onto $\PP^k(F)$ for all $F\in\FF_H$.
The fact that the triple $(K,\,\calU^{m,k}(K),\,\hSigma_K)$ defines a finite element is a consequence of the fact that the dimensions of $\calU^{m,k}(K)$ and $\hU_K^{m,k}$ coincide, and of the following important property (which states the existence of a right inverse for $\hSigma_K$).

\begin{lemma}[Reconstruction] \label{lem:rec}
The reconstruction operator $r_K$ satisfies
$\hSigma_K(r_K(\hv_K))=\hv_K$ for all $\hv_K \in \hU_K^{m,k}$, i.e.,
\begin{subequations} \label{eq_projection} \begin{alignat}{2}
\label{eq_projection_K}
(r_K(\hv_K),r)_K &= (v_K,r)_K &\quad &\forall r \in \PP^{m}(K)\,, \\
\label{eq_projection_F}
(r_K(\hv_K), q)_{\partial K} &= ( v_{\FF_K}, q)_{\partial K} &\quad &\forall q \in \PP^k(\FF_K)\,.
\end{alignat} \end{subequations}
\end{lemma}

\begin{proof}
We need to prove that
\[
\Theta \defi (r_K(\hv_K)-v_K,r)_K+(r_K(\hv_K)-v_{\FF_K}, q)_{\partial K}=0,
\]
for all $(r,q)\in \hU_K^{m,k}$. Let $\Phi_{r,q}\in \calU^{m,k}(K)$ solve
the following well-posed Neumann problem:
$-\div(\mathbb{A}\grad \Phi_{r,q})=r$ in $K$, and  
$\mathbb{A}\grad \Phi_{r,q\mid\partial K} \cdot \nn_{K} = q'$ on $\partial K$ with
$q'\defi q-\frac{1}{|\partial K|}((r,1)_{K}+(q,1)_{\partial K})$.
We observe that
\begin{align*}
\Theta &= (r_K(\hv_K)-v_K,r)_K+(r_K(\hv_K)-v_{\FF_K}, q')_{\partial K} \\
&=-(r_K(\hv_K)-v_K,\div(\mathbb{A}\grad \Phi_{r,q}))_K+(r_K(\hv_K)-v_{\FF_K},\mathbb{A}\grad \Phi_{r,q\mid\partial K} \cdot \nn_{K})_{\partial K}\\
&=(\mathbb{A}\grad r_K(\hv_K) , \grad \Phi_{r,q})_K+(v_K,\div(\mathbb{A}\grad \Phi_{r,q}))_K
-(v_{\FF_K},\mathbb{A}\grad \Phi_{r,q\mid\partial K} \cdot \nn_{K})_{\partial K} = 0,
\end{align*}
where we used \eqref{eq:recons.b} in the first line, the definition of 
$\Phi_{r,q}$ in the second line, and integration by parts (along with the symmetry of $\mathbb{A}$) together with
\eqref{eq:recons.a} with $w\defi\Phi_{r,q}$ in the third line.
\end{proof}

In the MsHHO method, the essential boundary conditions can be enforced strongly by considering
the subspace
\begin{equation}
\hU_{H,0}^{m,k} \defi\PP^{m}(\TT_H) \times \PP^k_0(\FF_H)\,.
\end{equation}
The MsHHO method for Problem~\eqref{eq_laplace_weak} reads as follows: Find $\hu_H\in\hU_{H,0}^{m,k}$ such that
\begin{equation} \label{mshho}
  \sum_{K \in \TT_H} (\mathbb{A}\grad r_K(\hu_K), \grad r_K(\hv_K))_K
=
\sum_{K \in \TT_H} (f_{K}, v_{K})_K \quad \forall \,\hv_H \in \hU_{H,0}^{m,k}\,.
\end{equation}
The approximate MsHHO solution $u_H\uHHO\in \calU^{m,k}(\TT_H)$ is then defined by
\begin{gather} \label{eq:def_uHHO_1}
u_{H\mid K}\uHHO\defi r_K(\hu_K) \quad\forall K\in\TT_H\,.
\end{gather}
It is easy to see that the function $u_H\uHHO$ defined in~\eqref{eq:def_uHHO_1} actually sits in $\tH^{1,k}_0(\TT_H)$. Indeed, owing to~\eqref{eq_projection_F}, for any interface $F \in \FFi_H$ such that $F \subseteq \partial K_+ \cap \partial K_-$, one has for all $q \in \PP^k(F)$,
\begin{align*}
(\jump{u_H\uHHO}_F,q)_F
&= (r_{K_+}(\hu_{K_+})\,(\nn_{K_+,F}\cdot\nn_F),q)_F
+
(r_{K_-}(\hu_{K_-})\,(\nn_{K_-,F}\cdot\nn_F),q)_F
\\
&= (u_{\FF_{K_+}}\,(\nn_{K_+,F}\cdot\nn_F),q)_F + (u_{\FF_{K_-}}\,(\nn_{K_-,F}\cdot\nn_F),q)_F
\\
&= (u_F\,(\nn_{K_+,F}\cdot\nn_F),q)_F + (u_F\,(\nn_{K_-,F}\cdot\nn_F),q)_F=0\,.
\end{align*}
For boundary faces, one uses again~\eqref{eq_projection_F} along with the fact that $\hu_H\in\hU_{H,0}^{m,k}$.
A crucial observation made in \cite[Remark~5.4]{CiErL:19a}, which is a direct consequence of the finite element property, is that 
the MsHHO method can be equivalently reformulated as follows: Find $u_H\uHHO \in \calU^{m,k}(\TT_H)\cap\tH^{1,k}_0(\TT_H)$ such that
\begin{equation}\label{eq:MsHHO_equiv}
(\mathbb{A} \grad_Hu_H\uHHO,\grad_Hv_H)_\Omega = (\Pi_H^{m}(f),v_H)_\Omega
\quad\forall \,v_H\in\calU^{m,k}(\TT_H)\cap\tH^{1,k}_0(\TT_H)\,,
\end{equation}
where, for any $K\in\TT_H$, ${\Pi^{m}_H(f)}_{\mid K}\defi\Pi^{m}_K(f_K)$.
The existence and uniqueness of $u_H\uHHO$ solution to the square system~\eqref{eq:MsHHO_equiv} 
is straightforward. Indeed,
if $\grad (u_{H\mid K}\uHHO)=\boldsymbol{0}$ in all
$K\in\TT_H$, then $u_H\uHHO\in\mathbb{P}^0(\TT_H)$, and since the moments of $u_H\uHHO$ are single-valued at the mesh interfaces and vanish at the mesh boundary faces, then $u_H\uHHO$ vanishes identically in $\Omega$.

\begin{lemma}[Characterization of the MsHHO solution]
\label{lemma_hho}
Let $u_H\uHHO$ solve~\eqref{eq:MsHHO_equiv}. Then, \textup{(i)}
$u_H\uHHO\in \calU^{m,k}(\TT_H)\cap\tH^{1,k}_0(\TT_H)$; 
\textup{(ii)} $u_H\uHHO \in \calV(\TT_H;\ddiv,\Omega)$ 
and $-\div(\mathbb{A}\grad_H u_H\uHHO) = \Pi^{m}_H(f)$ in $\Omega$.
\end{lemma}

\begin{proof}
We have already shown above that $u_H\uHHO\in \calU^{m,k}(\TT_H)\cap\tH^{1,k}_0(\TT_H)$.
Let us now show that $\mathbb{A}\grad_H u_H\uHHO\in \HH(\ddiv,\Omega)$.
Since $u_H\uHHO \in \calU^{m,k}(\TT_H)$, we already know that
$\div(\mathbb{A}\grad_H u_H\uHHO)_{\mid K} \in \PP^{m}(K)\subset L^2(K)$ and
${\mathbb{A}\grad_H u_H\uHHO}\,|_{\partial K} \cdot \nn_K \in \PP^k(\FF_K)$ for all
$K\in\TT_H$. Moreover, owing to~\eqref{mshho}, \eqref{eq:def_uHHO_1}, and the definition~\eqref{eq:recons}, we infer that 
\begin{equation} \label{form_flux}
-\sum_{K\in \TT_H} (\div(\mathbb{A}\grad_H u_H\uHHO),v_K)_K + \sum_{F \in \FFi_H} (\jump{\mathbb{A}\grad_H u_H\uHHO}_F \cdot \nn_F,v_F)_F
=
\sum_{K \in \TT_H}(\Pi^{m}_K(f_K), v_K)_K\,,
\end{equation}
for all $v_K\in\PP^{m}(K)$ and all $K\in\TT_H$, and for all $v_F\in\PP^k(F)$ and all $F\in\FFi_H$ (notice that
we have used that $v_F = 0$ for all $F\in\FFe_H$ for $\hv_H\in\hU_{H,0}^{m,k}$). This readily implies that 
\begin{gather*}
-{\div(\mathbb{A}\grad_H u_{H}\uHHO)}_{\mid K}=\Pi^{m}_K(f_K)\quad\text{for all $K\in\TT_H$}\,,
\end{gather*}
and that
\begin{equation*} 
  \jump{\mathbb{A}\grad_H u_H\uHHO}_F \cdot \nn_F = 0\quad\text{for all $F \in \FFi_H$}\,.
\end{equation*}
It follows that
$\mathbb{A}\grad_H u_H\uHHO \in \HH(\ddiv,\Omega)$
and that
$-\div (\mathbb{A}\grad_H u_H\uHHO) = \Pi^{m}_H(f)$ in $\Omega$.
\end{proof}

\section{Main equivalence result and further comments}
\label{equivalence}

The following result, which is a consequence of Lemma~\ref{lem:semi-mhm}, Lemma~\ref{lem:fully-mhm},
and Lemma~\ref{lemma_hho}, summarizes our main result
on the equivalence between the MHM and MsHHO methods.

\begin{theorem}[Equivalence between MHM and MsHHO] \label{th:equiv}
Let $m,k\in\mathbb{N}$. The following holds true:
\\
\textup{(i)} Let $u_H\uMHM$ be the (original, semi-explicit) MHM solution defined by~\eqref{mhm-sol} using $k\ge0$.
Let $u_H\uHHO$ be the MsHHO solution solving~\eqref{eq:MsHHO_equiv} using 
$m,k\ge0$. Then,
$u_H\uMHM=u_H\uHHO$ if $f \in \PP^{m}(\TT_H)$. 
\\
\textup{(ii)} Let $u_H\uMHM$ be the (fully explicit) MHM solution
defined by~\eqref{mhm-sol-alt} using $m,k\ge0$.
Let $u_H\uHHO$ be the MsHHO solution solving~\eqref{eq:MsHHO_equiv} using $m,k\ge0$. Then,
$u_H\uMHM=u_H\uHHO$ for all $f\in L^2(\Omega)$.
\end{theorem}

We now collect several remarks providing further insight into the above equivalence result.

\begin{remark}[Comparison of heuristic viewpoints] \label{re:heu}
It is possible 
to sketch the two complementary visions behind the fully explicit MHM
and MsHHO methods. In the (fully explicit) MHM method, the general idea is 
to search for an approximate solution $u_H$ among the members of the affine functional space
\begin{gather*}
\left\{v_H\in\VV(\TT_H;\ddiv,\Omega)\cap\calU^{m,k}(\TT_H)\,:\,-\div(\mathbb{A}\grad_H v_H)=\Pi^m_H(f)\;\text{in }\Omega \right\}\,,
\end{gather*}
and to enforce that $u_H\in\tH^{1,k}_0(\TT_H)$ by requiring that 
\begin{equation*} 
\langle \mu_H, u_H\rangle_{\dT_H}=0\quad\text{for all }\mu_H\in\Lambda^{k}(\dT_H)\,.
\end{equation*}
In the MsHHO method, the general idea is to search for an approximate solution among the members of the affine functional space
\begin{gather*}
\left\{v_H\in \tH^{1,k}_0(\TT_H)\cap\calU^{m,k}(\TT_H)\,:\,-\div(\mathbb{A}\grad (v_{H\mid K}))=\Pi^m_K(f_K)\text{ in $K$\; $\forall\,K\in\TT_H$}\right\}\,,
\end{gather*}
and to enforce that $u_H\in\VV(\TT_H;\ddiv,\Omega)$ by requiring that
\begin{equation*} 
\langle\mathbb{A}\grad_Hu_H\cdot\nn,q_H\rangle_{\dT_H}=0\quad\text{for all }q_H\in\mathbb{P}^k_0(\FF_H)\,.
\end{equation*}
\end{remark}

\begin{remark}[Modification of the right-hand side] \label{re:rhs}
It is observed in \cite[Remark~5.8]{CiErL:19a} that a variant of the MsHHO method is obtained by searching
$u_H\uHHO \in \calU^{m,k}(\TT_H)\cap\tH^{1,k}_0(\TT_H)$ such that
\begin{equation}\label{eq:MsHHO_modif}
(\mathbb{A} \grad_Hu_H\uHHO,\grad_Hv_H)_\Omega = (f,v_H)_\Omega
\quad\forall \,v_H\in \calU^{m,k}(\TT_H)\cap\tH^{1,k}_0(\TT_H)\,.
\end{equation}
One advantage of~\eqref{eq:MsHHO_modif} is that the source term $f$ is now seen through its 
$L^2$-orthogonal projection onto $\calU^{m,k}(\TT_H)$ instead of its projection onto the smaller space 
$\PP^{m}(\TT_H)$ as in~\eqref{eq:MsHHO_equiv}. However, if $u_H\uHHO$ solves~\eqref{eq:MsHHO_modif},
$\mathbb{A}\grad_Hu_H\uHHO$ slightly departs from $\HH(\ddiv,\Omega)$, i.e., we no longer 
have $u_H\uHHO\in \calV(\TT_H;\ddiv,\Omega)$ as for the solution to~\eqref{eq:MsHHO_equiv}. This modified MsHHO solution can be bridged to the fully explicit MHM
solution obtained by approximating the lifting $T\us$ by the operator $T\us_H:L^2(\Omega) \to \calU^{m,k}(\TT_H)^\perp$ such that, for all $g\in L^2(\Omega)$, $T\us_H(g)\in\calU^{m,k}(\TT_H)^\perp$ solves
\begin{equation*}
(\mathbb{A}\grad_H T\us_H(g),\grad_H v)_\Omega = (g,v),\qquad \forall v\in \calU^{m,k}(\TT_H)^\perp\,.
\end{equation*}
Indeed, the modified MsHHO solution solving~\eqref{eq:MsHHO_modif} coincides with 
the fully explicit MHM solution 
\begin{equation*}
u_H\uMHM \defi u^0_H + T\uN(\lambda_H) + T\us_H(f) \,,
\end{equation*}
where $(u^0_H,\lambda_H) \in \PP^0(\TT_H) \times \Lambda^k(\dT_H)$ now solve
\begin{alignat*}{2}
\langle\lambda_H, v^0_H\rangle_{\dT_H} &= -(f,v^0_H)_{\Omega} & \quad &\forall v^0_H \in \PP^0(\TT_H)\,, \\
\langle\mu_H, u^0_H\rangle_{\dT_H} + \langle\mu_H, T\uN(\lambda_H)\rangle_{\dT_H}  &= -\langle\mu_H,T\us_H(f)\rangle_{\dT_h} & \quad & \forall \mu_H \in \Lambda^k(\dT_H)\,.
\end{alignat*}
\end{remark}

\begin{remark}[Variant with no cell unknowns (case $m=-1$)] \label{re:mo}
It is possible to consider the case $m=-1$ in the above MHM and MsHHO settings, leading to an MsHHO formulation without cell unknowns. 
The spaces $\calU^{m,q}(K)$ and $\calU^{m,q}(\TT_H)$ can still be defined by~\eqref{eq:def_calU} when $m=-1$, with the convention that $\PP^{-1}(K)\defi\{0\}$. 
The fully explicit MHM method is still defined as in Section~\ref{sec:mhm}. 
The only modification in the analysis is that the last statement in
Lemma~\ref{lem:fully-mhm} now becomes
$-\div(\mathbb{A}\grad_H u_H\uMHM) = \Pi^{0}_H(f)$ in $\Omega$. 
Notice also that \eqref{mhm-sol-alt} becomes $u_H\uMHM=u_H^0+T\uN(\lambda_H)$.
Actually, since $T\us(c_H)=0$ for any $c_H\in\mathbb{P}^0(\TT_H)$ owing to~\eqref{eq:prop_Tus}, we infer that the (fully explicit) MHM method for $m=-1$ coincides with the (fully explicit) MHM method for $m=0$.
Concerning the MsHHO method, the variant~\eqref{eq:MsHHO_modif} has to be adopted in the case $m=-1$.
Finally, we observe that in the case $m=-1$, the MHM and MsHHO solutions do not 
coincide.
\end{remark}

\section{Unified convergence analysis} \label{se:conv}

We derive, in a unified fashion, an energy-norm error estimate that is valid for both the (fully explicit) MHM and MsHHO methods.

\subsection{Setting} \label{sse:setting}

Let $\TT_H$ be a given (coarse) polytopal mesh of the domain $\Omega$ in the sense of Section~\ref{sse:part}. Since we are interested in deriving a quantitative estimate on the discretization error for the MHM/MsHHO methods, we need to define a measure of regularity for the mesh at hand. To do so, following~\cite[Sec.~2.1.1]{CiErP:21}, we assume that the mesh $\TT_H$ admits a matching simplicial submesh $\mathcal{S}_H$, and that there exists some real parameter $0<\rho_H<1$ such that, for all $K\in\TT_H$, and all $T\in\mathcal{S}_H$ such that $T\subseteq K$, \begin{inparaenum}[(i)] \item $\rho_HH_T\leq R_T$ where $R_T$ denotes the inradius of the simplex $T$, and \item $\rho_HH_K\leq H_T$. \end{inparaenum} The parameter $\rho_H$ measures the regularity of the mesh $\TT_H$. When studying a convergence process in which the meshes of some given sequence $(\TT_H)_{H\in\mathcal{H}}$ are successively refined, we shall assume that the mesh sequence $(\TT_H)_{H\in\mathcal{H}}$ is uniformly regular, in the sense that there exists $0<\rho<1$ such that, for all $H\in\mathcal{H}$, $\rho\leq\rho_H$. Standard local Poincar\'e--Steklov and (continuous) trace and inverse inequalities, as well as (optimal) approximation properties for local $L^2$-orthogonal polynomial projectors, then hold on each cell $K\in\TT_H$ for any $H\in\mathcal{H}$, with multiplicative constants only depending on $\rho$. We refer the reader, e.g., to~\cite{Brenn:03} for the idea of submeshing into simplices, to \cite[Sec.~1.4.3]{DiPEr:12} for the (continuous) trace and inverse inequalities, to \cite{VeeVe:12} and \cite[Lem.~5.7]{ErnGu:17} for Poincar\'e--Steklov inequalities on sets composed of simplices, and to \cite[Lem.~5.6]{ErnGu:17} for the resulting higher-order polynomial approximation properties; see also the recent monographs~\cite{DPDro:20,CiErP:21} on HHO methods. In what follows, we use the symbol $\lesssim$ to denote an inequality that is valid up to a multiplicative constant only depending on the discretization through the parameter $\rho$.

In order to track the dependency of the error estimates with respect to the diffusion coefficient, for any $K\in\TT_H$, we denote by $a_{\flat,K}>0$ the local smallest eigenvalue of the coefficient $\mathbb{A}$ in the cell $K$, in such a way that $\mathbb{A}(\boldsymbol{x})\boldsymbol{\xi}\cdot\boldsymbol{\xi}\geq a_{\flat,K}|\boldsymbol{\xi}|^2$ for all $\boldsymbol{\xi}\in\mathbb{R}^d$ and almost every $\boldsymbol{x}\in K$.

Finally, given any measurable set $D\subset\overline{\Omega}$, and any integer $s\geq 0$, we respectively denote by $|\cdot|_{s,D}$ and $\|\cdot\|_{s,D}$ the standard seminorm and norm in $H^s(D;\mathbb{R}^\ell)$, for $\ell\in\{1,d\}$. We also define $H^s(\TT_H;\mathbb{R}^\ell)$ as the space of piecewise $\mathbb{R}^\ell$-valued $H^s$ functions on the partition $\TT_H$, with the convention that $H^s(\TT_H;\mathbb{R})$ is simply noted $H^s(\TT_H)$.

\subsection{Local approximation}

Let $m,k\in\mathbb{N}$ be given. Let $K\in\TT_H$, and recall the definition~\eqref{eq:def_calU} of the space $\calU^{m,k}(K)$.

\begin{lemma}[Approximation in $\calU^{m,k}(K)$] \label{lem:appr}
  Let $v\in H^1(K)$, and set $g\defi -\div(\mathbb{A}\grad v)$ in $K$.
  Assume that $g\in H^{m+1}(K)$ and that $\mathbb{A}\grad v\in H^{k+1}(K;\mathbb{R}^d)$.
  There exists $\pi^{m,k}_K(v)\in\calU^{m,k}(K)$ such that
  \begin{equation} \label{eq:interpol}
    \|\mathbb{A}^{\nicefrac12}\grad\big(v-\pi_K^{m,k}(v)\big)\|_{0,K}\lesssim a_{\flat,K}^{-\nicefrac12}\left(H_K^{m+2}|g|_{m+1,K}+H_K^{k+1}|\mathbb{A}\grad v|_{k+1,K}\right).
  \end{equation}
\end{lemma}
\begin{proof}
  Define $\pi^{m,k}_K(v)\in\calU^{m,k}(K)$ such that
  \begin{equation} \label{inter}
    -\div(\mathbb{A}\grad\pi^{m,k}_K(v))=\Pi^m_K(g)\quad\text{in }K,\qquad\mathbb{A}\grad\pi^{m,k}_K(v)\cdot\boldsymbol{n}_K=\Pi^k_{\FF_K}(\mathbb{A}\grad v\cdot\boldsymbol{n}_K)\quad\text{on }\partial K.
  \end{equation}
  Since $g=-\div(\mathbb{A}\grad v)$, we easily check that $(\Pi^m_K(g),1)_K+(\Pi^k_{\FF_K}(\mathbb{A}\grad v\cdot\boldsymbol{n}_K),1)_{\partial K}=0$; hence, the data of the Neumann problem~\eqref{inter} are compatible, and $\pi^{m,k}_K(v)$ is well-defined (up to an additive constant). Multiplying the first relation in~\eqref{inter} by $w\in H^1(K)$, integrating by parts, and using the compatibility of the data, yields
  \begin{equation} \label{calc.1}
    \begin{split}
    (\mathbb{A}\grad\pi^{m,k}_K(v),\grad w)_K&=(\Pi^m_K(g),w)_K+(\Pi^k_{\FF_K}(\mathbb{A}\grad v\cdot\boldsymbol{n}_K),w)_{\partial K}\\
      &=(\Pi^m_K(g),w-\Pi^0_K(w))_K+(\Pi^k_{\FF_K}(\mathbb{A}\grad v\cdot\boldsymbol{n}_K),w-\Pi^0_K(w))_{\partial K}.
      \end{split}
  \end{equation}
  By definition of $g$, we also have
  \begin{equation} \label{calc.2}
    \begin{split}
    (\mathbb{A}\grad v,\grad w)_K&=(g,w)_K+(\mathbb{A}\grad v\cdot\boldsymbol{n}_K,w)_{\partial K}\\
      &=(g,w-\Pi^0_K(w))_K+(\mathbb{A}\grad v\cdot\boldsymbol{n}_K,w-\Pi^0_K(w))_{\partial K}.
      \end{split}
  \end{equation}
  Subtracting~\eqref{calc.2} from~\eqref{calc.1}, we obtain, for any $w\in H^1(K)$,
  \begin{multline} \label{calc.3}
    (\mathbb{A}\grad\big(v-\pi^{m,k}_K(v)\big),\grad w)_K=(g-\Pi^m_K(g),w-\Pi^0_K(w))_K\\+(\mathbb{A}\grad v\cdot\boldsymbol{n}_K-\Pi^k_{\FF_K}(\mathbb{A}\grad v\cdot\boldsymbol{n}_K),w-\Pi^0_K(w))_{\partial K}.
  \end{multline}
  Applying the Cauchy--Schwarz inequality together with a local Poincar\'e--Steklov inequality for the first term in the right-hand side of~\eqref{calc.3}, and the Cauchy--Schwarz inequality combined with a (continuous) trace inequality and a local Poincar\'e--Steklov inequality for the second, we infer
  \begin{multline} \label{calc.4}
    (\mathbb{A}\grad\big(v-\pi^{m,k}_K(v)\big),\grad w)_K \lesssim \|g-\Pi^m_K(g)\|_{0,K}H_K|w|_{1,K}\\+\color{red}  \color{black}\|\mathbb{A}\grad v-\boldsymbol{\Pi}^k_{\FF_K}(\mathbb{A}\grad v)\|_{0,\partial K}H_K^{\nicefrac12}|w|_{1,K},
  \end{multline}
  where we also used the fact that $\Pi^k_{\FF_K}(\mathbb{A}\grad v\cdot\boldsymbol{n}_K)=\boldsymbol{\Pi}^k_{\FF_K}(\mathbb{A}\grad v)\cdot\boldsymbol{n}_K$ since the mesh faces are planar, combined with the fact that $\boldsymbol{n}_K$ is unitary, to handle the boundary term. By definition of $L^2$-orthogonal projectors, we have
  \begin{equation} \label{partial.opt}
    \|\mathbb{A}\grad v-\boldsymbol{\Pi}^k_{\FF_K}(\mathbb{A}\grad v)\|_{0,\partial K}=\min_{\boldsymbol{p}\in\mathbb{P}^k(\FF_K;\mathbb{R}^d)}\|\mathbb{A}\grad v-\boldsymbol{p}\|_{0,\partial K}\leq\|\mathbb{A}\grad v-\boldsymbol{\Pi}^k_K(\mathbb{A}\grad v)\|_{0,\partial K}.
  \end{equation}
  By standard approximation properties of $L^2$-orthogonal projectors, we finally obtain from~\eqref{calc.4} and~\eqref{partial.opt},
  $$\sup_{w\in H^1(K)\setminus\{0\}}\frac{(\mathbb{A}\grad\big(v-\pi^{m,k}_K(v)\big),\grad w)_K}{|w|_{1,K}}\lesssim H_K^{m+2}|g|_{m+1,K}+H_K^{k+1}|\mathbb{A}\grad v|_{k+1,K}.$$
  The conclusion follows choosing $w=v-\pi^{m,k}_K(v)$, and since $|w|_{1,K}^2\leq a_{\flat,K}^{-1}\|\mathbb{A}^{\nicefrac12}\grad w\|_{0,K}^2$.
\end{proof}

\begin{remark}[Case $m=-1$] \label{memo}
  Recall that $\mathbb{P}^{-1}(K)\defi\{0\}$. The result of Lemma~\ref{lem:appr} remains valid as it is in the case $m=-1$ (for $g\in L^2(K)$). The proof needs just be slightly adapted with respect to the general case $m\geq 0$. The interpolant $\pi^{-1,k}(v)\in\calU^{-1,k}(K)$ is defined as follows:
  \begin{equation*}
    -\div(\mathbb{A}\grad\pi^{-1,k}_K(v))=0\quad\text{in }K,\qquad\mathbb{A}\grad\pi^{-1,k}_K(v)\cdot\boldsymbol{n}_K=\Pi^k_{\FF_K}(\mathbb{A}\grad v\cdot\boldsymbol{n}_K)+\frac{1}{|\partial K|}(g,1)_{K}\quad\text{on }\partial K.
  \end{equation*}
  The identity~\eqref{calc.3} becomes
  \begin{multline*}
    (\mathbb{A}\grad\big(v-\pi^{-1,k}_K(v)\big),\grad w)_K=(g,w-\Pi^0_K(w))_K-\frac{1}{|\partial K|}(g,1)_K(w-\Pi^0_K(w),1)_{\partial K}\\+(\mathbb{A}\grad v\cdot\boldsymbol{n}_K-\Pi^k_{\FF_K}(\mathbb{A}\grad v\cdot\boldsymbol{n}_K),w-\Pi^0_K(w))_{\partial K}.
  \end{multline*}
  The conclusion then follows from the same arguments, using in addition that $\frac{|K|}{|\partial K|}\lesssim H_K$ under our mesh regularity assumptions to handle the second term in the first line of the right-hand side. 
\end{remark}

\subsection{Energy-norm error estimate}

Let $m,k\in\mathbb{N}$ be given.
We introduce, for any $K\in\TT_H$, the (local, canonical) interpolation operator $\mathcal{I}_K:H^1(K)\to\calU^{m,k}(K)$ associated with the finite element~\eqref{eq:triple_MsHHO} such that $\mathcal{I}_K\defi r_K\circ\widehat{\Sigma}_K$. Using the definition~\eqref{eq:recons} of the reconstruction operator, as well as the definition of the reduction operator $\widehat{\Sigma}_K$, we infer that, for any $v\in H^1(K)$,
\begin{subequations} \label{eq:ell.proj} \begin{align}
(\mathbb{A}\grad \mathcal{I}_K(v) , \grad w)_K
&=(\mathbb{A}\grad v,\grad w)_K
\quad \forall w \in \calU^{m,k}(K)\,, \label{epa}\\
(\mathcal{I}_K(v),1)_{\partial K} &= (v,1)_{\partial K}\,.
\end{align}\end{subequations}
Hence, $\mathcal{I}_K(v)\in\calU^{m,k}(K)$ is the ($\mathbb{A}$-weighted) elliptic projection of $v\in H^1(K)$ onto $\calU^{m,k}(K)$. As such, it satisfies
\begin{equation} \label{eq:opt}
  \|\mathbb{A}^{\nicefrac12}\grad\big(v-\mathcal{I}_K(v)\big)\|_{0,K}=\min_{w\in\calU^{m,k}(K)}\|\mathbb{A}^{\nicefrac12}\grad\big(v-w)\|_{0,K}.
\end{equation}
\begin{theorem}[Energy-norm error estimate] \label{th:err.est}
  Recall that $u\in H^1_0(\Omega)$ is the unique solution to~\eqref{eq_laplace_weak}. Let $u_H\in\calU^{m,k}(\TT_H)\cap\tH^{1,k}_0(\TT_H)$ denote either the (fully explicit) MHM solution~\eqref{mhm-sol-alt} to Problem~\eqref{mhm-alt}, or the MsHHO solution~\eqref{eq:def_uHHO_1} to Problem~\eqref{mshho}. Assume that $f\in H^{m+1}(\TT_H)$ and that $\mathbb{A}\grad u\in H^{k+1}(\TT_H;\mathbb{R}^d)$. Then, we have
  \begin{equation} \label{eq:estim}
    \|\mathbb{A}^{\nicefrac12}\grad_H\big(u-u_H\big)\|_{0,\Omega}\lesssim\left(\sum_{K\in\TT_H}a_{\flat,K}^{-1}\left(H_K^{2(m+2)}|f|_{m+1,K}^2+H_K^{2(k+1)}|\mathbb{A}\grad u|_{k+1,K}^2\right)\right)^{\nicefrac12}.
  \end{equation}
\end{theorem}
\begin{proof}
  First, by Theorem~\ref{th:equiv}, we know that the fully explicit MHM and MsHHO solutions coincide for all $f\in L^2(\Omega)$. We consider here the characterization~\eqref{eq:MsHHO_equiv} of $u_H$. Let $\mathcal{I}_H:H^1(\TT_H)\to\calU^{m,k}(\TT_H)$ denote the global interpolation operator such that, for all $v\in H^1(\TT_H)$, $\mathcal{I}_H(v)_{\mid K}\defi\mathcal{I}_K(v_{K})$ for all $K\in\TT_H$. Remark that, since $u\in H^1_0(\Omega)$, $\mathcal{I}_H(u)\in\calU^{m,k}(\TT_H)\cap\tH^{1,k}_0(\TT_H)$. By the triangle inequality, we split the discretization error as follows:
  \begin{equation} \label{split}
    \|\mathbb{A}^{\nicefrac12}\grad_H\big(u-u_H\big)\|_{0,\Omega}\leq\|\mathbb{A}^{\nicefrac12}\grad_H\big(u-\mathcal{I}_H(u)\big)\|_{0,\Omega}+\|\mathbb{A}^{\nicefrac12}\grad_H\big(\mathcal{I}_H(u)-u_H\big)\|_{0,\Omega}.
  \end{equation}
  The first term in the right-hand side of~\eqref{split} is an approximation error, and is estimated using the optimality property~\eqref{eq:opt} combined with the local approximation properties in $\calU^{m,k}(\TT_H)$ of Lemma~\ref{lem:appr}. Letting, for all $v\in H^1(\TT_H)$, $\pi^{m,k}_H(v)\in\calU^{m,k}(\TT_H)$ be the global interpolate such that $\pi^{m,k}_H(v)_{\mid K}=\pi^{m,k}_K(v_K)$ for all $K\in\TT_H$, we infer
  \begin{equation} \label{appro}
    \begin{split}
      \|\mathbb{A}^{\nicefrac12}\grad_H\big(u-\mathcal{I}_H(u)\big)\|_{0,\Omega}&=\min_{w_H\in\calU^{m,k}(\TT_H)}\|\mathbb{A}^{\nicefrac12}\grad_H\big(u-w_H\big)\|_{0,\Omega}\\&\leq\|\mathbb{A}^{\nicefrac12}\grad_H\big(u-\pi^{m,k}_H(u)\big)\|_{0,\Omega}\\&\lesssim\left(\sum_{K\in\TT_H}a_{\flat,K}^{-1}\left(H_K^{2(m+2)}|f|_{m+1,K}^2+H_K^{2(k+1)}|\mathbb{A}\grad u|_{k+1,K}^2\right)\right)^{\nicefrac12}.
    \end{split}
  \end{equation}
  The second term in the right-hand side of~\eqref{split} is the consistency error of the method, which satisfies, since $\big(\mathcal{I}_H(u)-u_H\big)\in\widetilde{\calU}^{m,k}_0(\TT_H)\defi\calU^{m,k}(\TT_H)\cap\tH^{1,k}_0(\TT_H)$,
  \begin{equation} \label{max}
    \|\mathbb{A}^{\nicefrac12}\grad_H\big(\mathcal{I}_H(u)-u_H\big)\|_{0,\Omega}=\max_{\substack{v_H\in\widetilde{\calU}_0^{m,k}(\TT_H),\\\|\mathbb{A}^{\nicefrac12}\grad_Hv_H\|_{0,\Omega}=1}}(\mathbb{A}\grad_H\big(\mathcal{I}_H(u)-u_H\big),\grad_Hv_H)_{\Omega}.
  \end{equation}
  Let $v_H\in\widetilde{\calU}_0^{m,k}(\TT_H)$ be such that $\|\mathbb{A}^{\nicefrac12}\grad_Hv_H\|_{0,\Omega}=1$. Since $u_H$ solves~\eqref{eq:MsHHO_equiv}, we infer
  \begin{equation}\label{dev}
    \begin{split}
      (\mathbb{A}\grad_H\big(\mathcal{I}_H(u)-u_H\big),\grad_Hv_H)_{\Omega}&=(\mathbb{A}\grad_H\mathcal{I}_H(u),\grad_Hv_H)_{\Omega}-(\Pi^m_H(f),v_H)_{\Omega}\\
      &=(\mathbb{A}\grad_H\mathcal{I}_H(u),\grad_Hv_H)_{\Omega}+(\div(\mathbb{A}\grad u),v_H)_{\Omega}+(f-\Pi^m_H(f),v_H)_{\Omega}\\
      &=(\mathbb{A}\grad_H\big(\mathcal{I}_H(u)-u\big),\grad_Hv_H)_{\Omega}+\sum_{K\in\TT_H}\sum_{F\in\FF_K}(\mathbb{A}\grad u_K\cdot\boldsymbol{n}_{K,F},v_K)_F\\&\hspace{1cm}+(f-\Pi^m_H(f),v_H)_{\Omega}\\
      &=\sum_{F\in\FF_H}(\mathbb{A}\grad u\cdot\boldsymbol{n}_F,\llbracket v_H\rrbracket_F)_F+(f-\Pi^m_H(f),v_H)_{\Omega}\ifed\mathfrak{T}_1+\mathfrak{T}_2\,,
    \end{split}
  \end{equation}
  where we added and subtracted $(f,v_H)_{\Omega}$ and used the fact that $f=-\div(\mathbb{A}\grad u)$ in $\Omega$ to pass from the first to the second line, we performed cell-by-cell integration by parts to pass from the second to the third line, and finally used the local orthogonality property~\eqref{epa} as well as the fact that $\llbracket\mathbb{A}\grad u\rrbracket_F\cdot\boldsymbol{n}_F=0$ for all $F\in\FFi_H$ as a consequence of the fact that $\mathbb{A}\grad u\in\boldsymbol{H}({\rm div},\Omega)\cap H^1(\TT_H;\mathbb{R}^d)$ to pass from the third to the fourth line.
  To estimate $\mathfrak{T}_1$, we remark that, since $v_H\in\tH^{1,k}_0(\TT_H)$, $\Pi^k_F(\llbracket v_H\rrbracket_F)=0$ for all $F\in\FF_H$. We thus have
  \begin{equation*}
    \begin{split}
      \mathfrak{T}_1&=\sum_{F\in\FF_H}(\mathbb{A}\grad u\cdot\boldsymbol{n}_F-\Pi^k_F(\mathbb{A}\grad u\cdot\boldsymbol{n}_F),\llbracket v_H-\Pi^0_F(v_H)\rrbracket_F)_F\\&=\sum_{K\in\TT_H}\sum_{F\in\FF_K}(\big(\mathbb{A}\grad u_K-\boldsymbol{\Pi}^k_F(\mathbb{A}\grad u_K)\big)\cdot\boldsymbol{n}_{K,F},v_K-\Pi^0_F(v_K))_F\,.
    \end{split}
  \end{equation*}
  By two successive applications of the Cauchy--Schwarz inequality, we infer
  \begin{equation*}
    \mathfrak{T}_1\leq\left(\sum_{K\in\TT_H}a_{\flat,K}^{-1}H_K\|\mathbb{A}\grad u_K-\boldsymbol{\Pi}^k_{\FF_K}(\mathbb{A}\grad u_K)\|^2_{0,\partial K}\right)^{\nicefrac12}\left(\sum_{K\in\TT_H}a_{\flat,K}H_K^{-1}\|v_K-\Pi^0_{\FF_K}(v_K)\|^2_{0,\partial K}\right)^{\nicefrac12}.
  \end{equation*}
  The first factor in the right-hand side is estimated using~\eqref{partial.opt} and standard approximation properties of $L^2$-orthogonal projectors. The second factor is estimated by adding/subtracting $\Pi^0_K(v_K)$, using a triangle inequality combined with the $L^2(\partial K)$-stability of $\Pi^0_{\FF_K}$, and concluding by the use of a (continuous) trace inequality combined with a local Poincar\'e--Steklov inequality. We obtain
  \begin{equation*}
    \mathfrak{T}_1\lesssim \left(\sum_{K\in\TT_H}a_{\flat,K}^{-1}H_K^{2(k+1)}|\mathbb{A}\grad u|_{k+1,K}^2\right)^{\nicefrac12}\left(\sum_{K\in\TT_H}a_{\flat,K}|v_K|_{1,K}^2\right)^{\nicefrac12}.
  \end{equation*}
  Recalling that $\|\mathbb{A}^{\nicefrac12}\grad_Hv_H\|_{0,\Omega}=1$, and since $a_{\flat,K}|v_K|_{1,K}^2\leq\|\mathbb{A}^{\nicefrac12}\grad v_K\|_{0,K}^2$, we finally infer that
  \begin{equation}\label{t1}
    \mathfrak{T}_1\lesssim \left(\sum_{K\in\TT_H}a_{\flat,K}^{-1}H_K^{2(k+1)}|\mathbb{A}\grad u|_{k+1,K}^2\right)^{\nicefrac12}.
  \end{equation}
  The term $\mathfrak{T}_2$ is, in turn, easily estimated using the definition of the $L^2$-orthogonal projection to write
  $$\mathfrak{T}_2=(f-\Pi^m_H(f),v_H-\Pi^0_H(v_H))_{\Omega},$$
  and invoking the Cauchy--Schwarz inequality, a local Poincar\'e--Steklov inequality, and standard approximation properties of $L^2$-orthogonal projectors to conclude. We obtain
  \begin{equation}\label{t2}
    \begin{split}
      \mathfrak{T}_2&\lesssim \left(\sum_{K\in\TT_H}a_{\flat,K}^{-1}H_K^{2(m+2)}|f|_{m+1,K}^2\right)^{\nicefrac12}\left(\sum_{K\in\TT_H}a_{\flat,K}|v_K|_{1,K}^2\right)^{\nicefrac12}\\&\lesssim \left(\sum_{K\in\TT_H}a_{\flat,K}^{-1}H_K^{2(m+2)}|f|_{m+1,K}^2\right)^{\nicefrac12},
    \end{split}
  \end{equation}
  where we used again that $\|\mathbb{A}^{\nicefrac12}\grad_Hv_H\|_{0,\Omega}=1$ to pass from the first to the second line.
  Finally, plugging~\eqref{t1}-\eqref{t2}-\eqref{dev}-\eqref{max} and~\eqref{appro} into~\eqref{split} proves~\eqref{eq:estim}.
\end{proof}

\begin{remark}[Case $m=-1$] \label{rem:mo}
  We know from Remark~\ref{re:mo} that the (fully explicit) MHM method for $m=-1$ coincides with the (fully explicit) MHM method for $m=0$. As far as the MsHHO method is concerned, in the case $m=-1$, one adopts the variant~\eqref{eq:MsHHO_modif} of the method, and the a priori estimate of Theorem~\ref{th:err.est} remains valid as is (for $f\in L^2(\Omega)$). The proof actually simplifies with respect to the general case $m\geq 0$, since the term $\mathfrak{T}_2$ can be discarded. The conclusion follows from Lemma~\ref{lem:appr} and Remark~\ref{memo}.
\end{remark}

\begin{remark}[Case $m=k-1$]
  In the case $m=k-1$, the result~\eqref{eq:estim} (see Remark~\ref{rem:mo} for the case $k=0$ and $m=-1$) simplifies since $|f|_{k,K}\leq\sqrt{d}\,|\mathbb{A}\grad u|_{k+1,K}$ for all $K\in\TT_H$. Under the sole assumption that 
$\mathbb{A}\grad u\in H^{k+1}(\TT_H;\mathbb{R}^d)$, we then have
  $$\|\mathbb{A}^{\nicefrac12}\grad_H\big(u-u_H\big)\|_{0,\Omega}\lesssim\left(\sum_{K\in\TT_H}a_{\flat,K}^{-1}H_K^{2(k+1)}|\mathbb{A}\grad u|_{k+1,K}^2\right)^{\nicefrac12}.$$
  In the MHM setting, when $k=0$ (then one can discard the contribution given by the operator $T\us$), we obtain an optimal error estimate under the sole assumption on the source term that $f\in L^2(\Omega)$, which improves on~\cite[Corollary 4.2]{ArHPV:13} where more regularity is needed.
\end{remark}

\begin{remark}[Link with previous results]
  In the MHM framework, the error estimate of Theorem~\ref{th:err.est} is a refined version of~\cite[Theorem 4.1]{ArHPV:13} (for the original, semi-explicit MHM method), both in terms of regularity assumptions and in terms of tracking of the dependency of the multiplicative constants with respect to the diffusion coefficient. In the MsHHO framework, such an error estimate is new, and is complementary to the homogenization-based error estimate of~\cite[Theorem 5.6]{CiErL:19a} (such a homogenization-based analysis is also available in the MHM setting; cf.~\cite{PaVaV:17}). The a priori estimate of~\cite[Theorem 5.6]{CiErL:19a} is robust in highly oscillatory diffusion regimes but is suboptimal for mildly varying diffusion. The present result fills this gap. 
\end{remark}

\section{Basis functions and solution strategies}
\label{basis-ddm}

We address the decomposition of the MHM and MsHHO solutions in terms of multiscale basis functions and highlight the impact of such a decomposition on the possible organization of the computations using an offline-online strategy.
Let $k\geq 1$ be a given integer. In what follows, to keep the presentation simple, we consider for a polynomial degree $k$ 
on the faces
the polynomial degree $m\defi k-1\geq 0$ in the cell, and, following our convention, we simply write $\calU^{k}(K)$ in place of $\calU^{k-1,k}(K)$ for all $K\in\TT_H$.
The key observation is that there are two possible constructions of basis functions for the local space $\calU^{k}(K)$. Both sets of basis functions are composed of cell-based and face-based functions. The construction of the two sets is however different. The first construction, referred to as {\em primal set}, will prove to be relevant for the MHM method, whereas the second construction, referred to as {\em dual set}, will prove to be relevant for the MsHHO method.

\subsection{Basis functions} \label{sse:basis}
 
\subsubsection{Polynomial basis functions}

Let $q\in\mathbb{N}$. We denote by $n^q_l$ the dimension of the vector space of $l$-variate
polynomial functions of total degree up to $q$. For any cell $K\in\TT_H$,
let $\{\psi^{q,K}_{i}\}_{1\leq i\leq n^q_d}$ be a basis of
$\PP^q(K)$, and for any face $F\in\FF_H$, let
$\{\psi^{q,F}_{j}\}_{1\leq j\leq n^q_{d{-}1}}$ be a basis of
$\PP^q(F)$.
With the choice of degree $q\defi k-1$ in the cell and degree $q\defi k$ 
on the faces, we henceforth drop
the corresponding superscripts in the polynomial basis functions
to alleviate the notation. For convenience, we assume that 
$\psi^{K}_{1}\equiv 1$; this assumption will be useful in the MHM setting.

\subsubsection{Primal basis functions}

For $K\in\TT_H$, we locally construct the set of primal basis functions for $\calU^{k}(K)$.
Regarding the cell-based basis functions, we set $\phi^{\pb,K}_{1}\equiv 1$, and for all $2\leq i\leq n^{k-1}_d$, we define $\phi^{\pb,K}_i$ as the unique function in $H^1(K)^\perp$ solving the following well-posed Neumann problem:
\begin{equation}
\label{phiK}
\left\{\begin{aligned}
-\div(\mathbb{A}\grad \phi^{\pb,K}_i)&=\psi^{K}_{i}-\Pi^0_K(\psi^{K}_{i})\;\text{in $K$}\,, \\
\mathbb{A}\grad\phi^{\pb,K}_i\cdot\nn_K&=0\;\text{on $\partial K$}\,.
\end{aligned}\right.
\end{equation}
Concerning the face-based basis functions, for all $F\in\FF_K$ and all $1\leq j\leq n^{k}_{d{-}1}$, we define $\phi^{\pb,K}_{F,j}$ as the unique function in $H^1(K)^\perp$ solving the following well-posed Neumann problem:
\begin{equation}
\label{phiF}
\left\{\begin{aligned}
  -\div(\mathbb{A}\grad\phi^{\pb,K}_{F,j}) & =-\frac{1}{|K|}(\psi^F_{j},1)_F\;\text{in $K$}\,, \\
  \mathbb{A}\grad\phi^{\pb,K}_{F,j}\cdot\nn_{K,F} & =\psi^{F}_{j}\;\text{on $F$}\quad\text{and}\quad \mathbb{A}\grad\phi^{\pb,K}_{F,j}\cdot\nn_{K,\sigma}=0\;\text{on $\sigma\in\FF_K\setminus\{F\}$}\,.
\end{aligned}\right.
\end{equation}
Then, for all $v\in\calU^{k}(K)$, setting
\begin{enumerate}
\item[(i)] $-\div(\mathbb{A}\grad v)\defi g_K=g_{K,1}+\sum_{i=2}^{n^{k-1}_d}g_{K,i}\psi^{K}_{i}\in\PP^{k-1}(K)$ (recall that $\psi^{K}_{1}\equiv 1$)\,,
\item[(ii)] $\mathbb{A}\grad v\,|_{\partial K}\cdot\nn_K\defi\mu_{\FF_K}\in \PP^{k}(\FF_K)$ with $\mu_{\FF_K\mid F}=\sum_{j=1}^{n^k_{d-1}}\mu_{F,j}\psi^F_{j}$ for all $F\in\FF_K$,
\item[(iii)] $\Pi^0_K(v)\defi v^0_K\in \PP^0(K)$\,,
\end{enumerate} 
with $(g_K,1)_K+(\mu_{\FF_K},1)_{\partial K}=0$, we have
\begin{gather}
v=v^0_K+\sum_{F\in\FF_K}\sum_{j=1}^{n^{k}_{d{-}1}}\mu_{F,j}\phi^{\pb,K}_{F,j}+\sum_{i=2}^{n^{k-1}_d}g_{K,i}\phi^{\pb,K}_i\,.
\end{gather}

A set of global basis functions for the space $\calU^k(\TT_H)\cap\calV(\TT_H;\ddiv,\Omega)$ is given by
\begin{gather*}
\{\widetilde{\phi}^{\pb,K}_{i}\}_{K\in\TT_H,1\leq i\leq n^{k-1}_d}\cup\{\widetilde{\phi}^{\pb,F}_{j}\}_{F\in\FF_H,1\leq j\leq n^k_{d{-}1}}\,,
\end{gather*}
where for each cell $K\in\TT_H$,
\begin{gather}
\widetilde{\phi}^{\pb,K}_{i}\,|_{K}=\phi^{\pb,K}_{i}\quad\text{and}\quad\widetilde{\phi}^{\pb,K}_{i}\,|_{\Omega\setminus\overline{K}}=0\,,
\end{gather}
for each interface $F\subseteq\partial K_+\cap\partial K_-$, 
\begin{gather}
\widetilde{\phi}^{\pb,F}_{j}\,|_{K_\pm}=(\nn_{K_\pm,F}\cdot\nn_F)\phi^{\pb,K_\pm}_{F,j}\quad\text{and}\quad\widetilde{\phi}^{\pb,F}_{j}\,|_{\Omega\setminus\overline{K_+\cup K_-}}=0\,,
\end{gather}
and for each boundary face $F\subseteq\partial K\cap\partial\Omega$,
\begin{gather}
\widetilde{\phi}^{\pb,F}_{j}\,|_{K}=\phi^{\pb,K}_{F,j}\quad\text{and}\quad\widetilde{\phi}^{\pb,F}_{j}\,|_{\Omega\setminus\overline{K}}=0\,.
\end{gather}

\begin{remark}[Link to lifting operators] \label{rem:lift}
Recall the local lifting operators $T_K\uN,T_K\us$ and their global counterparts $T\uN,T\us$ 
introduced in Section~\ref{sec:mhm}. For all $K\in\TT_H$, one readily verifies that
\begin{equation}
\phi^{\pb,K}_i = T_K\us(\psi_i^K), \qquad \phi^{\pb,K}_{F,j} = T_K\uN(E_F^{\dK}(\psi_j^F)),
\end{equation}
where the first identity holds for all $2\le i\le n_d^{k-1}$ and the
second identity holds for all $F\in\FF_K$ and all $1\le j\le n_{d-1}^k$, where $E_F^{\dK}$ 
denotes the zero-extension operator from $F$ to $\dK$. For the global basis functions, we have
\begin{equation}
\widetilde{\phi}^{\pb,K}_i = T\us(E_K^\Omega(\psi_i^K)), \qquad
\widetilde{\phi}^{\pb,F}_j = T\uN(E_F^{\dT_H}(\psi_j^F)),
\end{equation}
where $E_K^\Omega$ denotes the zero-extension operator from $K$ to $\Omega$, and
$E_F^{\dT_H}(\psi_j^F)\,|_{\dK} \defi E_F^{\partial K} (\psi_j^F(\nn_{K,F}\cdot\nn_F))$ if $F\in\FF_K$
and $E_F^{\dT_H}(\psi_j^F)\,|_{\dK} \defi 0$ otherwise, for all $K\in\TT_H$.
\end{remark}

\begin{remark}[Energy minimization]\label{rem:energy}
Consider the local energy functional
$J_K\,:\,H^1(K)\to\mathbb{R}_+$ such that 
$\varphi\mapsto\frac{1}{2}(\mathbb{A}\grad\varphi,\grad\varphi)_K$.
Then, one can characterize $\phi^{\pb,K}_i$ for all $2\leq i\leq n^{k-1}_d$ as follows:
\begin{equation} \label{eq:pb.c}
  \phi^{\pb,K}_i=\arg\min_{\varphi\in H^1(K)^\perp}\left(J_K(\varphi)-\big(\psi^{K}_{i}-\Pi^0_K(\psi^{K}_{i}),\varphi\big)_K\right)\,,
\end{equation}
and one can characterize $\phi^{\pb,K}_{F,j}$ for all $F\in\FF_K$ and all $1\leq j\leq n^{k}_{d{-}1}$ as follows:
\begin{equation} \label{eq:pb.f}
  \phi^{\pb,K}_{F,j}=\arg\min_{\varphi\in H^1(K)^\perp}\left(J_K(\varphi)-(\psi^F_{j},\varphi)_F+\frac{1}{|K|}(\psi^F_{j},1)_F(\varphi,1)_K\right)\,,
\end{equation}
where we recall that $H^1(K)^\perp \defi \left\{v\in H^1(K)\,:\,(v,1)_K=0\right\}$.
\end{remark}

\subsubsection{Dual basis functions}

For $K\in\TT_H$, we locally construct the set of dual basis functions for $\calU^{k}(K)$.
For this purpose, we rely on the fact that the triple $(K,\calU^k(K),\hSigma_K)$ is a 
finite element (see~\eqref{eq:triple_MsHHO}). 
For all $1\leq i\leq n^{k{-}1}_d$, the cell-based basis functions $\phi^{\db,K}_{i}\in\calU^k(K)$ are obtained by 
requiring that 
\begin{equation}  
\label{phiKd}
\Pi_K^{k{-}1}(\phi^{\db,K}_{i})=\psi^{K}_{i}\,,\quad\Pi_{\FF_K}^{k}(\phi^{\db,K}_{i})=0\,,
\end{equation}
that is, we have $\phi^{\db,K}_{i} \defi r_K((\psi^{K}_{i},0))$.
Moreover, for all $F\in\FF_K$ and all $1\le j\le n_{d-1}^k$, the face-based basis functions $\phi^{\db,K}_{F,j}\in\calU^k(K)$ are obtained by requiring that
\begin{equation}
\label{phiFd}
\Pi_K^{k{-}1}(\phi^{\db,K}_{F,j})=0,\quad\Pi_F^{k}(\phi^{\db,K}_{F,j})=\psi^{F}_{j},\quad\Pi_\sigma^{k}(\phi^{\db,K}_{F,j})=0\text{ for all }\sigma\in\FF_K\setminus\{F\}\,,
\end{equation}
that is, we have $\phi^{\db,K}_{F,j}\defi r_K((0,E_F^{\partial K}(\psi_j^F)))$.
Then, for all $v\in\calU^k(K)$, setting
\begin{enumerate}
\item[(i)] $\Pi_K^{k{-}1}(v)\defi v_K=\sum_{i=1}^{n^{k{-}1}_d} v_{K,i}\psi^{K}_{i}\in\PP^{k{-}1}(K)$,
\item[(ii)] $\Pi_{\FF_K}^k(v)\defi v_{\FF_K}\in\PP^{k}(\FF_K)$ with
$v_{\FF_K\mid F}= \sum_{j=1}^{n^k_{d{-}1}} v_{F,j}\psi^F_{j}$ for all $F\in\FF_K$,
\end{enumerate}
we have
\begin{gather}
v=\sum_{i=1}^{n^{k{-}1}_d} v_{K,i}\phi^{\db,K}_{i}+\sum_{F\in\FF_K}\sum_{j=1}^{n^k_{d{-}1}} v_{F,j}\phi^{\db,K}_{F,j}\,.
\end{gather}
Notice that we also have $v=r_K(\hv_K)$ where 
$\hv_K\defi(v_K,v_{\FF_K})\in\hU_K^k$.

A set of global basis functions for the space $\calU^k(\TT_H)\cap\tH^{1,k}_0(\TT_H)$ is given by
\begin{gather}\label{bas.da}
\{\widetilde{\phi}^{\db,K}_{i}\}_{K\in\TT_H,1\leq i\leq n^{k{-}1}_d}\cup\{\widetilde{\phi}^{\db,F}_{j}\}_{F\in\FFi_H,1\leq j\leq n^k_{d{-}1}}\,,
\end{gather}
where for each cell $K\in\TT_H$,
\begin{gather}\label{bas.db}
\widetilde{\phi}^{\db,K}_{i}\,|_{K}=\phi^{\db,K}_{i}\quad\text{and}\quad\widetilde{\phi}^{\db,K}_{i}\,|_{\Omega\setminus\overline{K}}=0\,,
\end{gather}
and for each interface $F\subseteq\partial K_+\cap\partial K_-$,
\begin{gather}\label{bas.dc}
\widetilde{\phi}^{\db,F}_{j}\,|_{K_\pm}=\phi^{\db,K_\pm}_{F,j}\quad\text{and}\quad\widetilde{\phi}^{\db,F}_{j}\,|_{\Omega\setminus\overline{K_+\cup K_-}}=0\,.
\end{gather}

\begin{remark}[Energy minimization] \label{re:ener}
Recall the local energy functional $J_K$ defined in Remark~\ref{rem:energy}.
Then, one can characterize $\phi^{\db,K}_i$ for all $1\leq i\leq n^{k-1}_d$ as follows:
\begin{equation} \label{eq:db.c}
  \phi^{\db,K}_{i}\defi\arg\min_{\varphi\in H^K_i}J_K(\varphi)\,,
\end{equation}
where $H^K_i\defi \left\{v\in H^K\,:\,\Pi_K^{k{-}1}(v)=\psi^{K}_{i}\right\}$ is a 
nonempty, convex, closed subset of the Hilbert space
$H^K\defi\left\{v\in H^1(K)\,:\,\Pi^{k}_{\FF_K}(v)=0\right\}$. 
This means that $\phi^{\db,K}_{i}\in H^1(K)$ is obtained by
solving the following saddle-point problem with dual unknowns 
$\gamma^K_i\in\PP^{k{-}1}(K)$ and $\mu^{\dK}_i\in\PP^{k}(\FF_K)$ such that $(\gamma^K_i,1)_K+(\mu^{\dK}_i,1)_{\partial K}=0$:
\begin{equation}  \label{comp.dK}
\left\{\begin{aligned}
&-\div(\mathbb{A}\grad\phi^{\db,K}_{i})=\gamma^K_i\;\text{ in $K$}\,,\quad\mathbb{A}\grad\phi^{\db,K}_{i}\cdot\nn_K=\mu^{\dK}_i\text{ on }\partial K\,,\\
&\Pi_K^{k{-}1} (\phi^{\db,K}_{i})=\psi^{K}_{i}\,,\quad\Pi_{\FF_K}^{k}(\phi^{\db,K}_{i})=0\,.
\end{aligned}\right.
\end{equation}
Similarly, one can characterize $\phi^{\db,K}_{F,j}$ for all $F\in\FF_K$ and all $1\leq j\leq n^{k}_{d-1}$ as follows:
\begin{equation} 
\label{eq:db.f}
  \phi^{\db,K}_{F,j}\defi\arg\min_{\varphi\in H^K_{F,j}}J_K(\varphi)\,,
\end{equation}
where $H^K_{F,j}\defi \left\{v\in H^K_F\,:\,\Pi_F^{k}(v)=\psi^F_{j}\right\}$
is a 
nonempty, convex, closed subset of the Hilbert space $H^K_F\defi\left\{v\in H^1(K)\,:\,\Pi_K^{k{-}1}(v)= 0\text{ and } \Pi_\sigma^{k}(v)=0\quad\forall\,\sigma\in\FF_K\setminus\{F\}\right\}$. This means that $\phi^{\db,K}_{F,j}\in H^1(K)$ is obtained by 
solving the following saddle-point problem with dual unknowns 
$\gamma^K_{F,j}\in\PP^{k{-}1}(K)$ and $\mu^{\dK}_{F,j}\in\PP^{k}(\FF_K)$ such that $(\gamma^K_{F,j},1)_K+(\mu^{\dK}_{F,j},1)_{\partial K}=0$:
\begin{equation} \label{comp.dFj}
\left\{\begin{aligned}
&-\div(\mathbb{A}\grad\phi^{\db,K}_{F,j})=\gamma^K_{F,j}\;\text{ in $K$},\quad\mathbb{A}\grad\phi^{\db,K}_{F,j}\,\cdot\nn_K=\mu^{\dK}_{F,j}\text{ on }\partial K\,,\\ 
&\Pi_K^{k{-}1}(\phi^{\db,K}_{F,j})=0, \quad\Pi_F^{k}(\phi^{\db,K}_{F,j})=\psi^{F}_{j},\quad\Pi_\sigma^{k}(\phi^{\db,K}_{F,j})=0\text{ for all }\sigma\in\FF_K\setminus\{F\}\,.
\end{aligned}\right.
\end{equation}
\end{remark}

\subsection{Offline-online strategy}

In view of Section~\ref{sse:basis}, primal basis functions, as they globally span $\calU^k(\TT_H)\cap\calV(\TT_H;\ddiv,\Omega)$, appear to be naturally suited to the MHM framework. On the other hand, dual basis functions, as they globally span $\calU^k(\TT_H)\cap\tH^{1,k}_0(\TT_H)$, appear to be naturally suited to the MsHHO framework (cf.~Remark~\ref{re:heu}).
In this section, we detail how the MHM and MsHHO computations can be optimally organized using an offline-online strategy.
This type of organization of the computations is particularly relevant in a multi-query context, in which the solution has to be computed for a large amount of data, so that it is crucial to
pre-process as many data-independent quantities as possible in an offline stage, while keeping the size of the online system to its minimum. We focus in the sequel on the situation where many instances of the source term $f$ are considered (we could also consider the case of multiple boundary data).

\subsubsection{The MHM case} \label{ssse:MHM}

By Remark~\ref{rem:lift}, the (fully explicit) MHM solution $u_H\uMHM\in\calU^k(\TT_H)\cap\mathcal{V}(\TT_H;{\rm div},\Omega)$ defined by~\eqref{mhm-sol-alt} with $m\defi k-1$, where the pair $(u^0_H,\lambda_H) \in \PP^0(\TT_H) \times \Lambda^k(\dT_H)$ solves \eqref{mhm-alt}, writes
\begin{gather}
\label{uMHM}
u_H\uMHM=\sum_{K\in\TT_H}u^0_K\widetilde{\phi}^{\pb,K}_{1}+\sum_{F\in\FF_H}\sum_{j=1}^{n^k_{d{-}1}}\lambda_{F,j}\widetilde{\phi}^{\pb,F}_{j}+\sum_{K\in\TT_H} \sum_{i=2}^{n^{k-1}_d}f_{K,i}\widetilde\phi^{\pb,K}_{i},
\end{gather}
where $u^0_K\defi u_{H\mid K}^0= \Pi_K^0(u_H\uMHM)$ for all $K\in\TT_H$, $\lambda_{F,j}$ is defined, for all $F\in\FF_H$, as the $j^{\text{th}}$ coefficient of $\lambda_{H\mid F}$ on the basis $\{\psi^F_{j}\}_{1\leq j\leq n^k_{d{-}1}}$, and 
$f_{K,i}$ stands for the $i^{\text{th}}$ coefficient of $\Pi_K^{k-1}(f_K)$ on the basis $\{\psi^{K}_{i}\}_{1\leq i\leq n^{k-1}_{d}}$.
This motivates the following offline-online decomposition of the computations:

\medskip
\noindent
\underline{Offline stage:} For each $K\in\TT_H$:
\begin{itemize}
\item[$(1)$] Compute the basis functions 
$\phi^{\pb,K}_{i}$ from \eqref{phiK}, for  all $i=2,\ldots,n^{k-1}_{d}$;
\item[$(2)$] Compute the basis functions 
$\phi^{\pb,K}_{F,j}$ from \eqref{phiF}, for all $F\in\FF_K$ and all $j=1,\ldots,n^k_{d{-}1}$.
\end{itemize}
\underline{Online stage:}
\begin{itemize}
\item[$(3)$] Compute the vector $\left(f_{K,i}\right)_{K\in \TT_H}^{i=1,\ldots,n^{k-1}_{d}}$ by solving the local symmetric positive-definite (SPD) systems
  \begin{equation*}
    \sum_{i=1}^{n_d^{k-1}}f_{K,i}(\psi^K_i,\psi^K_j)_K=(f_K,\psi^K_j)_K\,,
  \end{equation*}
  for all $j=1,\ldots,n_d^{k-1}$, and all $K\in\TT_H$;
\item[$(4)$] Compute the vectors $\left(u_K^0\right)_{K\in\TT_H}$ and $\left(\lambda_{F,j}\right)_{F\in \FF_H}^{j=1,\ldots,n^k_{d{-}1}}$ by solving the global saddle-point problem
\begin{equation*}
\sum_{F\in\FF_K}\sum_{j=1}^{n^k_{d{-}1}}\lambda_{F,j}(\psi^{F}_{j},1)_{F}  = -(f_K,1)_K\,,
\end{equation*}
for all $K\in\TT_H$, and (recall that $\phi^{\pb,K}_{1}\equiv 1$ and that $(\phi^{\pb,K}_{F',j'},1)_{K}=0$)
\begin{equation*}
\sum_{K\in\TT_{F'}}u^0_K(\psi^{F'}_{j'},1)_{F'} + \sum_{K\in\TT_{F'}}\sum_{F\in\FF_K}\sum_{j=1}^{n^k_{d{-}1}}\lambda_{F,j}(\psi^{F'}_{j'},\widetilde{\phi}^{\pb,F}_{j\mid K})_{F'}  = -\sum_{K\in\TT_{F'}}\sum_{i=2}^{n^{k-1}_d}f_{K,i}(\psi^{K}_{i},\phi^{\pb,K}_{F',j'})_{K}\,,
\end{equation*}
for all $j'=1,\ldots,n^k_{d{-}1}$, and all $F'\in \FF_H$ with $\TT_{F'}\defi\{K_+,K_-\}$ if $F'\in\FFi_H$ and $\TT_{F'}\defi\{K\}$ if $F'\in\FFe_H$;
\item[$(5)$] Form $u_H\uMHM$ using \eqref{uMHM}.
\end{itemize}

\begin{remark}[Mono-query case] \label{rem:mono} In a mono-query scenario, in which the solution to the discrete problem is only needed for one (or a few) source term(s), one can advantageously consider an amended version of~\eqref{uMHM}, where the last term in the decomposition is simply replaced by $T\us(\Pi^{k-1}_H(f))$. From a practical point of view, the step (1) above can be bypassed, and replaced by solving, inbetween steps (3) and (4), Problem~\eqref{phiK} for all $K\in\TT_H$ with right-hand side $\Pi^{k-1}_K(f_K)$ (in place of $\psi_i^K$), whose solution is precisely $T\us_K(\Pi^{k-1}_K(f_K))$.
\end{remark}

\subsubsection{The MsHHO case}

The solution $u_H\uHHO\in\calU^k(\TT_H)\cap\tH^{1,k}_0(\TT_H)$ to Problem~\eqref{eq:MsHHO_equiv} writes
\begin{gather}
\label{uMsHHO}
u_H\uHHO=\sum_{K\in\TT_H}\sum_{i=1}^{n^{k{-}1}_d} u_{K,i}\widetilde{\phi}^{\db,K}_{i}+\sum_{F\in\FFi_H}\sum_{j=1}^{n^k_{d{-}1}} u_{F,j}\widetilde{\phi}^{\db,F}_{j}\,,
\end{gather}
where $u_{K,i}$ is defined as the $i^{\text{th}}$ coefficient of $u_K\defi \Pi_K^{k-1}(u_H\uHHO)$ on the basis $\{\psi^{K}_{i}\}_{1\leq i\leq n^{k{-}1}_d}$ for all $K\in\TT_H$, and $u_{F,j}$ as the $j^{\text{th}}$ coefficient of $u_F\defi \Pi^k_F(u_H\uHHO)$ on the basis $\{\psi^F_{j}\}_{1\leq j\leq n^k_{d{-}1}}$ for all $F\in\FFi_H$ (recall that $\Pi_F^k(u_H\uHHO)=0$ for all $F\in \FFe_H$). This, combined with the equivalent formulation~\eqref{form_flux} of the MsHHO method, and Remark~\ref{re:ener} (recall, in particular, the notation introduced therein), motivates the following offline-online decomposition of the computations:

\medskip
\noindent
\underline{Offline stage:} For each $K\in\TT_H$:
\begin{itemize}
\item[$(1)$] Compute the basis functions $\phi^{\db,K}_{i}$ from \eqref{comp.dK}, for all  $i=1,\ldots,n^{k-1}_{d}$;
\item[$(2)$] Compute the basis functions $\phi^{\db,K}_{F,j}$  from \eqref{comp.dFj}, for all $F\in\FF_K$ and all $j=1,\ldots,n^k_{d{-}1}$.
\end{itemize}
  Define
    \begin{itemize}
      \item[$\bullet$] the $n_d^{k-1}\times n_d^{k-1}$ matrix $\mathbb{G}^{KK}$, whose column $1\leq i\leq n^{k-1}_d$ is formed by the $n^{k-1}_d$ coefficients of the decomposition of $\gamma_i^K\in\mathbb{P}^{k-1}(K)$ on the basis $\{\psi^K_{i'}\}_{1\leq i'\leq n^{k-1}_d}$;
      \item[$\bullet$] for each $F\in\FF_K$, the $n^{k-1}_d\times n^k_{d-1}$ matrix $\mathbb{G}^{KF}$, whose column $1\leq j\leq n^{k}_{d-1}$ is formed by the $n^{k-1}_d$ coefficients of the decomposition of $\gamma^K_{F,j}\in\mathbb{P}^{k-1}(K)$ on the basis $\{\psi^K_i\}_{1\leq i\leq n^{k-1}_d}$;
      \item[$\bullet$] for each $F\in\FF_K$, the $n^{k}_{d-1}\times n^{k-1}_{d}$ matrix $\mathbb{M}^{FK}$, whose column $1\leq i\leq n^{k-1}_{d}$ is formed by the $n^{k}_{d-1}$ coefficients of the decomposition of $\mu^{\partial K}_{i\mid F}\in\mathbb{P}^{k}(F)$ on the basis $\{\psi^F_j\}_{1\leq j\leq n^{k}_{d-1}}$;
      \item[$\bullet$] for each $F,F'\in\FF_K$, the $n^{k}_{d-1}\times n^{k}_{d-1}$ matrix $\mathbb{M}^{F'F}$, whose column $1\leq j\leq n^{k}_{d-1}$ is formed by the $n^{k}_{d-1}$ coefficients of the decomposition of $\mu^{\partial K}_{F,j\mid F'}\in\mathbb{P}^{k}(F')$ on the basis $\{\psi^{F'}_{j'}\}_{1\leq j'\leq n^{k}_{d-1}}$;
    \end{itemize}
    \begin{itemize}
      \item[$(3)$] Invert the matrix $\mathbb{G}^{KK}$.
\end{itemize}
\underline{Online stage:}
\begin{itemize}
  \item[$(4)$] Compute the vectors $\left(\boldsymbol{f}_K\right)_{K\in\TT_H}\defi\left(f_{K,i}\right)_{K\in \TT_H}^{i=1,\ldots,n^{k-1}_{d}}$ by solving the local SPD systems
    \begin{equation*}
      \sum_{i=1}^{n_d^{k-1}}f_{K,i}(\psi^K_i,\psi^K_j)_K=(f_K,\psi^K_j)_K\,,
    \end{equation*}
    for all $j=1,\ldots,n_d^{k-1}$, and all $K\in\TT_H$;
  \item[$(5)$]  Compute the vectors $\left(\boldsymbol{u}_F\right)_{F\in\FFi_H}\defi\left(u_{F,j}\right)_{F\in \FFi_H}^{j=1,\ldots,n^k_{d{-}1}}$ by solving the global SPD problem
\begin{equation*}
\begin{aligned}
\sum_{K\in\TT_{F'}}\sum_{F\in\FF_K\cap\FFi_H}\big(\mathbb{M}^{F'F}-\mathbb{M}^{F'K}[\mathbb{G}^{KK}]^{-1}\mathbb{G}^{KF}\big)\boldsymbol{u}_{F}=-\sum_{K\in\TT_{F'}}\mathbb{M}^{F'K}[\mathbb{G}^{KK}]^{-1}\boldsymbol{f}_K,
\end{aligned}
\end{equation*}
for all $F'\in\FFi_H$;
\item[$(6)$] Reconstruct locally the vectors $\left(\boldsymbol{u}_K\right)_{K\in\TT_H}\defi\left(u_{K,i}\right)_{K\in \TT_H}^{i=1,\ldots,n^{k-1}_{d}}$: for all $K\in\TT_H$,
  \begin{equation*}
    \boldsymbol{u}_K=[\mathbb{G}^{KK}]^{-1}\bigg(\boldsymbol{f}_K-\sum_{F\in\FF_K\cap\FFi_H}\mathbb{G}^{KF}\boldsymbol{u}_F\bigg);
  \end{equation*}
\item[$(7)$] Form $u_H\uHHO$ using \eqref{uMsHHO}.
\end{itemize}

\subsubsection{Purely face-based MsHHO method} \label{sec:face-based}
    Using the (primal-dual) local set of basis functions for $\calU^k(K)$, $K\in\TT_H$, introduced in~\cite[Sec.~4.1]{CiErL:19a} (but not fully exploited therein), the MsHHO method can be alternatively defined as a purely face-based method, i.e.~without using cell unknowns. To see this, let $K\in\TT_H$, and recall the local energy functional $J_K$ defined in Remark~\ref{rem:energy}. Define $\phi_i^K$ for all $1\leq i\leq n^{k-1}_d$ as follows:
    \begin{equation} \label{eq:b.c}
      \phi_{i}^K\defi\arg\min_{\varphi\in H^K}\left(J_K(\varphi)-(\psi^{K}_{i},\varphi)_K\right)\,,
    \end{equation}
    where the space $H^K$ is defined in Remark~\ref{re:ener}. Equivalently, $\phi^K_i\in H^1(K)$ is obtained by solving the following saddle-point problem with dual unknown $\mu^{\dK}_i\in\PP^{k}(\FF_K)$ such that $(\psi^K_i,1)_K+(\mu^{\dK}_i,1)_{\partial K}=0$:
\begin{equation}  \label{comp.b.c}
\left\{\begin{aligned}
&-\div(\mathbb{A}\grad\phi^{K}_{i})=\psi^K_i\;\text{ in $K$}\,,\quad\mathbb{A}\grad\phi^{K}_{i}\cdot\nn_K=\mu^{\dK}_i\text{ on }\partial K\,,\\
&\Pi_{\FF_K}^{k}(\phi^{K}_{i})=0\,.
\end{aligned}\right.
\end{equation}
  Similarly, define $\phi^{K}_{F,j}$ for all $F\in\FF_K$ and all $1\leq j\leq n^{k}_{d-1}$ as follows:
\begin{equation} 
\label{eq:b.f}
  \phi^{K}_{F,j}\defi\arg\min_{\varphi\in H^K_{F,j}}J_K(\varphi)\,,
\end{equation}
where $H^K_{F,j}\defi \left\{v\in H^K_F\,:\,\Pi_F^{k}(v)=\psi^F_{j}\right\}$ as in Remark~\ref{re:ener}, but now we set $H^K_F\defi\big\{v\in H^1(K)\,:\,\Pi_\sigma^{k}(v)=0\quad\forall\,\sigma\in\FF_K\setminus\{F\}\big\}$. Equivalently, $\phi^{K}_{F,j}\in H^1(K)$ is obtained by solving the following saddle-point problem with dual unknown $\mu^{\dK}_{F,j}\in\PP^{k}(\FF_K)$ such that $(\mu^{\dK}_{F,j},1)_{\partial K}=0$:
\begin{equation} \label{comp.db.f}
\left\{\begin{aligned}
&-\div(\mathbb{A}\grad\phi^{K}_{F,j})=0\;\text{ in $K$},\quad\mathbb{A}\grad\phi^{K}_{F,j}\,\cdot\nn_K=\mu^{\dK}_{F,j}\text{ on }\partial K\,,\\ 
&\Pi_F^{k}(\phi^{K}_{F,j})=\psi^{F}_{j},\quad\Pi_\sigma^{k}(\phi^{K}_{F,j})=0\text{ for all }\sigma\in\FF_K\setminus\{F\}\,.
\end{aligned}\right.
\end{equation}
  For all $v\in\calU^k(K)$, setting
  \begin{enumerate}
\item[(i)] $-\div(\mathbb{A}\grad v)\defi g_K=\sum_{i=1}^{n^{k-1}_d}g_{K,i}\psi^K_i\in\PP^{k-1}(K)$,
\item[(ii)] $\Pi_{\FF_K}^k(v)\defi v_{\FF_K}\in\PP^{k}(\FF_K)$ with
$v_{\FF_K\mid F}= \sum_{j=1}^{n^k_{d{-}1}} v_{F,j}\psi^F_{j}$ for all $F\in\FF_K$,
  \end{enumerate}
  we then have
  \begin{gather}
v=\sum_{i=1}^{n^{k-1}_d}g_{K,i}\phi^{K}_i+\sum_{F\in\FF_K}\sum_{j=1}^{n^{k}_{d{-}1}}v_{F,j}\phi^{K}_{F,j}\,.
  \end{gather}
  As we did for the dual set of basis functions in~\eqref{bas.da}--\eqref{bas.db}--\eqref{bas.dc}, we can easily construct a set of global basis functions $\{\widetilde{\phi}^{K}_{i}\}_{K\in\TT_H,1\leq i\leq n^{k-1}_d}\cup\{\widetilde{\phi}^{F}_{j}\}_{F\in\FFi_H,1\leq j\leq n^{k}_{d-1}}$ for the space $\calU^k(\TT_H)\cap\tH^{1,k}_0(\TT_H)$. The solution $u_H\uHHO\in\calU^k(\TT_H)\cap\tH^{1,k}_0(\TT_H)$ to Problem~\eqref{eq:MsHHO_equiv} then writes
\begin{gather}
\label{uMsHHO.alt}
u_H\uHHO=\sum_{K\in\TT_H}\sum_{i=1}^{n^{k{-}1}_d} f_{K,i}\widetilde{\phi}^{K}_{i}+\sum_{F\in\FFi_H}\sum_{j=1}^{n^k_{d{-}1}} u_{F,j}\widetilde{\phi}^{F}_{j}\,,
\end{gather}
where $f_{K,i}$ is defined as the $i^{\text{th}}$ coefficient of $\Pi_K^{k-1}(f_K)$ on the basis $\{\psi^{K}_{i}\}_{1\leq i\leq n^{k{-}1}_d}$ for any $K\in\TT_H$, and $u_{F,j}$ as the $j^{\text{th}}$ coefficient of $u_F\defi \Pi^k_F(u_H\uHHO)$ on the basis $\{\psi^F_{j}\}_{1\leq j\leq n^k_{d{-}1}}$ for any $F\in\FFi_H$. The new decomposition~\eqref{uMsHHO.alt} leads to a simplification of the offline-online solution strategy. In the offline stage, the static condensation step (3) can be bypassed. Also, the steps (1) and (2), which consist in solving saddle-point problems of the form~\eqref{comp.b.c} and~\eqref{comp.db.f}, are a bit less expensive than before, as the number of Lagrange multipliers is decreased. In the online stage, the reconstruction step (6) can be bypassed, and the global problem to solve in the step (5) simplifies to finding $\left(\boldsymbol{u}_F\right)_{F\in\FFi_H}\defi\left(u_{F,j}\right)_{F\in \FFi_H}^{j=1,\ldots,n^k_{d{-}1}}$ such that
\begin{equation} \label{sys}
\begin{aligned}
\sum_{K\in\TT_{F'}}\sum_{F\in\FF_K\cap\FFi_H}\mathbb{M}^{F'F}\boldsymbol{u}_{F}=-\sum_{K\in\TT_{F'}}\mathbb{M}^{F'K}\boldsymbol{f}_K,
\end{aligned}
\end{equation}
for all $F'\in\FFi_H$.

  \begin{remark}[Mono-query case]
The purely face-based version of the MsHHO method is particularly suited to the mono-query context. In that case, the step (1) can be bypassed, and replaced by solving, inbetween steps (4) and (5), Problem~\eqref{comp.b.c} for all $K\in\TT_H$ with right-hand side $\Pi^{k-1}_K(f_K)$ (in place of $\psi^K_i$), whose solution is denoted $\phi_{f_K}^K$. Letting $\mu_{f_K}^{\partial K}$ be the corresponding dual unknown, one must then replace in~\eqref{sys} the vector $\mathbb{M}^{F'K}\boldsymbol{f}_K$ by the vector $\boldsymbol{\mu}^{\partial K}_{f_K,F'}\in\mathbb{R}^{n^k_{d-1}}$ formed by the coefficients of the decomposition of $\mu^{\partial K}_{f_K\mid F'}\in\mathbb{P}^k(F')$ on the basis $\{\psi_j^{F'}\}_{1\leq j\leq n^k_{d-1}}$. The MsHHO solution is now given by
\begin{equation}
  u_H\uHHO=\sum_{K\in\TT_H}\widetilde{\phi}^{K}_{f_K}+\sum_{F\in\FFi_H}\sum_{j=1}^{n^k_{d{-}1}} u_{F,j}\widetilde{\phi}^{F}_{j}\,,
\end{equation}
in place of~\eqref{uMsHHO.alt}.
\end{remark}

\subsubsection{Summary}

The following table summarizes the main computational aspects, in a multi-query context, for both the (fully explicit) MHM and MsHHO methods based on $\calU^k(\TT_H)$, $k\geq 1$, in both the offline and online stages, so as to provide to the reader a one-glance comparison of the two methods. For simplicity, we assume that all the mesh cells have the same number of faces, denoted by $n_{\partial}$.

\begin{center}\begin{table}[h!]\begin{tabular}{|l|l|l|l|}
\hline
MHM&offline&local SPD systems&$n_d^{k-1}-1 + n^k_{d-1}n_{\partial}$ problems per cell\\
&online&global saddle-point problem&$\#\TT_H+n^k_{d-1}\#\FF_H$ unknowns\\
\hline
MsHHO&offline&local saddle-point systems&$n^{k-1}_d+n^k_{d-1}n_{\partial}$ problems per cell\\
&online&global SPD problem&$n^k_{d-1}\#\FFi_H$ unknowns\\
\hline
\end{tabular}\caption{Comparison of MHM and MsHHO on the main computational aspects}\end{table}\end{center}

The offline stage is of course performed once and for all, independently of the data (here, the source term). In practice, for both methods, the approximation of 
the local problems can be computationally costly, but the fact that all problems are local makes of the offline stage an embarassingly parallel task. The offline stage can hence naturally benefit from parallel architectures. In the online stage, the linear systems to solve (for the different data) only attach unknowns to the coarse mesh at hand, hence the computational burden remains limited.

\begin{remark}[Other boundary conditions]
  The MHM and MsHHO methods easily adapt to the case of (nonhomogeneous) mixed Dirichlet--Neumann boundary conditions. If $\FF_H^{\rm D}\cup\FF^{\rm N}_H$ forms a (disjoint) partition of $\FFe_H$ into, respectively, Dirichlet and Neumann boundary faces, then the size of the online linear systems in the MHM method becomes $\#\TT_H+n^k_{d-1}\#(\FFi_H\cup\FF_H^{\rm D})$, whereas that for the MsHHO method becomes $n^k_{d-1}\#(\FFi_H\cup\FF_H^{\rm N})$.
\end{remark}

\begin{remark}[Second-level discretization and equivalence between two-level methods] \label{rem:sec_lev}
  Let $\mathcal{S}_h$ denote a matching simplicial submesh of $\TT_H$ of size $h\ll H$ ($\mathcal{S}_h$ can for example be obtained by further refining $\mathcal{S}_H$ from Section~\ref{sse:setting}). Consider, locally to any $K\in\TT_H$, a discretization of the second-level (Neumann) problems in the space $\calU^{m,k}(K_h)\cap\tH^{1,k}(K_h)$, where $K_h\defi\{T\}_{T\in\mathcal{S}_h,T\subset K}$. Then, using similar arguments as in the one-level case, one can prove the equivalence between the two-level MHM and MsHHO methods.
  Simple cases exist in which closed formulas for the second-level basis functions
    are available. For instance, if $T\in K_h$ is a simplex and $\mathbb{A}_{\mid T}$ is a constant matrix, we may cite the case
  $m=-1$ and $k=0$ for the MsHHO method where
  $\calU^{-1,0}(T)=\mathbb{P}^1(T)$, or the case $m=0$ and $k=0$ for the MHM/MsHHO methods
  where $\calU^{0,0}(T)$ corresponds to a proper subspace of $\mathbb{P}^2(T)$ if
  ${\mathbb A}_{\mid T}$ is isotropic (see \cite{HarParVal13}).
  Unfortunately, in general, even if $T\in K_h$
  is assumed to be a simplex and $\mathbb{A}_{\mid T}$ to be constant, closed-form expressions for basis functions in $\calU^{m,k}(T)$
  are not known.
  To recover equivalence for ready-to-use methods, one possibility is to write an HHO discretization of the second-level problems (as in~\cite{CiErL:19b}) and make the corresponding two-level MHM and MsHHO solutions coincide. In that case, the zero-jump condition on the normal flux at interfaces is imposed on a stabilized version of the normal flux (see~\cite{EfeLazShi15} for an example in the HDG setting). Notice that the subcells need not necessarily be simplices.
  It is also possible, at the price of equivalence, to preserve two-level $\HH(\ddiv,\Omega)$-conformity on the exact flux. This is the case in the MHM context as soon as a mixed method is used to approximate the second-level problems; see~\cite{DurDevGomVal19} (cf.~also \cite{VohWoh:13} for a similar idea in the context of mixed finite elements).
\end{remark}


\section{Conclusion}
\label{concl}

Although they originate from entirely different constructions, we have proved that the one-level
(original) semi-explicit  MHM method and the one-level MsHHO method provide the same numerical solution when the source term is piecewise polynomial on the (coarse) mesh, and this is also the case for the fully explicit MHM method and the MsHHO method for any source term in $L^2(\Omega)$. As a byproduct, we have proposed a unified convergence analysis, as well as improved versions
of the two methods. More precisely, we have introduced a version of the MHM method that is
prompt to be used in a multi-query context, and a version of the MsHHO method
that only uses face unknowns.

\bibliographystyle{amsplain}
\bibliography{mhmhho}

\end{document}